\documentclass[11pt]{amsart}
\usepackage{mathtools, amssymb, amsfonts, amsthm, enumitem}
\usepackage{epsfig, psfrag, color, multicol, graphicx, tikz,pdfsync}
\usepackage{amsfonts}
\usepackage{epstopdf}
\usepackage{xspace}

\usepackage{hyperref}

\usepackage[all]{xy}

\usetikzlibrary{calc}

\newtheorem{thm}{Theorem}[section]
\newtheorem{lemma}[thm]{Lemma}
\newtheorem{mainlem}[thm]{Main Lemma}

\newtheorem{cor}[thm]{Corollary}
\newtheorem{prop}[thm]{Proposition}

\theoremstyle{plain}

\newtheorem{lem}[thm]{Lemma}

\theoremstyle{definition}

\newtheorem{rem}[thm]{Remark}
\newtheorem{definition}[thm]{Definition}

\numberwithin{equation}{section}

\def\rr{\mathbb R}

\def\fq{{\mathbb F}_q}
\def\ov{\overline}

\def\sm{\smallsetminus}
\def\Ga{\Gamma}

\def\De{\Delta}

\def\si{\sigma}

\def\om{\omega}

\def\ve{\varepsilon}

\def\cg{\mathcal G}

\def\cp{\mathcal P}
\def\cP{\mathcal P}
\def\cq{\mathcal Q}

\def\ssu{\subset}

\def\<{\langle}
\def\>{\rangle}

\def\rM{ {\text {\rm {\bf M}}}}

\def\inv{{\text {\rm inv} } }

\def\0{{\mathbf 0}}

\def\.{\hskip.06cm}
\def\ts{\hskip.03cm}

\def\vol{{\text {\rm vol}}}

\def\dimo{{\diamond}}

\def\P{{\textup{\textsf{P}}}}

\def\SP{{\textup{\textsf{\#P}}}}
\def\NP{{\textup{\textsf{NP}}}}
\def\PH{{\textup{\textsf{PH}}}}

\def\N{\mathbb{N}}

\def\nin{\noindent}

\def\NP{{\textup{\textsf{NP}}}}

\renewcommand{\mod}[1]{
\;\, \textup{mod} \; #1
}

\def\ext{\boldsymbol{e}}
\def\phixz{\boxed{\phi}\left(x_i,z_i\right)}
\def\dtwoposet{\textsc{\#D2LE}\xspace}
\def\br{\textsc{\#Bruhat}\xspace}
\def\htwoposet{\textsc{\#H2LE}\xspace}
\def\incposet{\textsc{\#IPLE}}
\def\rigidcircuit{\textsc{\#RigidCircuit}\xspace}
\def\LE{\textsc{\#LE}\xspace}

\title{Counting linear extensions of restricted posets}
\date{}

\author[Samuel Dittmer \and Igor Pak]{Samuel Dittmer$^{\star}$ \ \. \and \ \. Igor~Pak$^{\star}$}

\thanks{\thinspace ${\hspace{-.45ex}}^\star$Department of Mathematics,
UCLA, Los Angeles, CA, 90095.}
\thanks{\thinspace \  Email:
\hskip.06cm  \texttt{\{samuel.dittmer,\ts{pak}\}@math.ucla.edu}}
\thanks{\thinspace \
\today}

\begin{document}

\begin{abstract}
The classical 1991 result by Brightwell and Winkler~\cite{BW1} states that
the number of linear extensions of a poset is \ts \SP-complete.  We extend this
result to posets with certain restrictions.  First,
we prove that the number of linear extension for
\emph{posets of height two} is \SP-complete. Furthermore,
we prove that this holds for \emph{incidence posets of graphs}.
Finally, we prove that the number of linear extensions for
\emph{posets of dimension two} is \ts \SP-complete.

\end{abstract}

\maketitle

\section{Introduction}

Counting \emph{linear extensions} (\LE) of a finite poset is a fundamental problem in
both Combinatorics and Computer Science, with connections and applications
ranging from Statistics to Optimization, to Social Choice Theory.  It is
primarily motivated by the following basic question: given a partial
information of preferences between various objects, what are
the chances of other comparisons?

In 1991, Brightwell and Winkler showed that \ts \LE \ts is \ts \SP-complete~\cite{BW1}, 
but for various restricted classes of posets the problem
remains unresolved.  Notably, they conjectured that the following problem
is \ts \SP-complete:

\medskip

\nin \hskip1.9cm {\large \htwoposet} (\emph{Number of linear extensions of height-2 posets})

\nin \hskip1.9cm \textbf{Input:} A partially ordered set~$P$ of height~$2$.

\nin \hskip1.9cm  \textbf{Output:} The number~$\ext(P)$ of linear extensions.

\medskip

\nin
Here \emph{height two} means that $P$ has two levels, i.e.\ no chains of length~$3$.
This problem has been open for 27 years, most recently reiterated
in~\cite{Hu2,LS}.  Its solution is the first result in this paper.

\smallskip

\begin{thm} \label{HeightTwoTheorem}
\htwoposet is \SP-complete.
\end{thm}

\smallskip

Our second result is an extension of Theorem~\ref{HeightTwoTheorem}.
It was proposed recently by Lee and Skipper in~\cite{LS},
motivated by non-linear combinatorial optimization.

\medskip

\nin \hskip1.9cm{\large \incposet} (\emph{Number of linear extensions of incidence posets})

\nin \hskip1.9cm\textbf{Input:} A graph~$G=(V,E)$.

\nin \hskip1.9cm \textbf{Output:} The number~$\ext(I_G)$ of linear extensions of the incidence poset~$I_G$.

\medskip

\nin
Here the incidence poset $I_G$ is defined as a height~2 posets with vertices~$V$ on one
level, edges~$E$ on another level, and the inequalities defined by adjacencies in~$G$.

\begin{thm} \label{IncidenceTheorem}
\incposet\ is \SP-complete.
\end{thm}

\smallskip

Theorem~\ref{IncidenceTheorem} implies Theorem~\ref{HeightTwoTheorem}, of course.
Formally, the proofs of both results are independent, but use the same technical
ideas of using number theory to to obtain targeted reductions modulo primes.
However, since the proof Theorem~\ref{HeightTwoTheorem} is both technically and
conceptually simpler, we chose to include both proofs.

\smallskip

Our main and the most difficult result is the solution of the following
natural problem posed in 1988 by M\"ohring~\cite[p.~163]{Moh}, and then
again in 1997 by Felsner and Wernisch~\cite{FW} motivated by different
applications.

\medskip

\nin \hskip1.9cm {\large \dtwoposet} (\emph{Number of linear extensions of dimension-2 posets})

\nin \hskip1.9cm\textbf{Input:} A partially ordered set~$P$ of dimension two.

\nin \hskip1.9cm \textbf{Output:} The number~$\ext(P)$ of linear extensions of~$P$.

\medskip

\nin
Here the poset $P$ is said to have \emph{dimension two} if it can be represented
by a finite set of points \ts $\bigl\{ (x_1,y_1), \ldots, (x_n,y_n)\bigr\} \ssu \rr^2$,
with the inequalities $(x_i,y_i) \preccurlyeq (x_j,y_j)$ if $x_i \le x_j$ and $y_i \le y_j$, $i\ne j$.
Equivalently, poset~$P$ has dimension two if and only if its comparability graph $\Ga(P)$ has complement
$\ov{\Ga(P)}\simeq\Ga(P^\ast)$, for a \emph{dual poset}~$P^\ast$
(see e.g.~\cite{Tro}).

\smallskip

\begin{thm} \label{Dim2Cor}
\dtwoposet is \SP-complete.
\end{thm}

\smallskip

As a motivation, Felsner and Wernisch~\cite{FW},
show that \dtwoposet is equivalent to the following problem on the number
of possible bubble sorted permutations~$\tau$ from a given $\si \in S_n$
(see also~\cite{BW2,Reu}).

\medskip

\nin \hskip1.9cm{\large \br} (\emph{Size of principal ideal in the weak Bruhat order})

\nin \hskip1.9cm\textbf{Input:} A permutation~$\sigma \in S_n$.

\nin \hskip1.9cm\textbf{Output:} The number $\ext(\sigma)$ of permutations~$\tau \in S_n$ with~$\tau \leq \sigma$.

\medskip

\nin
Here we write $\tau \leq \sigma$ if $\tau$ can be obtained from $\si$ by a \emph{bubble sorting}: repeated
application of adjacent transpositions which the minimal possible number of inversions:
$$
\sigma \, = \, \tau \ts \cdot \ts (i_1,i_1+1) \ts\ts \cdots \ts\ts (i_\ell,i_\ell+1), \quad \text{where} \ \ \,
\inv(\si) \. = \. \inv(\tau) \ts + \ts \ell\ts.
$$
The \emph{weak Bruhat order}~$B_n$ is defined to be $(S_n,\le)$.  
In \br, we consider the principal ideal $P_\si=B_n \cap \{\om\le \si\}$,  
so in the notation above $\ext(\si)=\ext(P_\si)$.  Note that \br is in~\SP.

We include a quick proof of the reduction of \dtwoposet to~\br in~$\S$\ref{ss:setup-LE}, 
both for completeness and to introduce the framework for the proof of the main result.

\begin{thm} \label{BruhatTheorem}
\br is \SP-complete.
\end{thm}

The proof of Theorem~\ref{BruhatTheorem} is presented in two stages. First, we will describe a combinatorial problem \textsc{\#RigidCircuit}. In Lemma~\ref{CircuitTo3Sat} we give a parsimonious reduction from \textsc{\#3SAT}, which is \SP-complete, to \textsc{\#RigidCircuits}. Then, in Lemma~\ref{BruhatToCircuit}, we use a more complicated set of reductions from \textsc{\#RigidCircuits} to \br to show that \br is \SP-complete.

Let us emphasize that the proof of Lemma~\ref{BruhatToCircuit} is \emph{computer assisted},
i.e.\ it has gates found by computer, but which in principle can be checked directly.
See $\S$\ref{ss:finrem-implem} for a detailed discussion of computational aspects of the proof.

\begin{rem}{\rm
The proof in~\cite{BW1} uses a modulo~$p$ argument and the Chinese Remainder Theorem,
which we also employ for all results (cf.~\ref{ss:finrem-complexity}).  In fact, 
this approach is one of the few applicable  for these problems, since the existence 
of FPRAS (see below) strongly suggests the impossibility of 
a parsimonious reduction of \textsc{\#3SAT} and its relatives.}
\end{rem}

\medskip

\subsection*{Historical review}
The notion of \emph{\SP-completeness} was introduced by Valiant~\cite{Val} a a way to 
characterize the class of computationally hard counting problems; see~\cite{MM,Pap} for a modern treatment.
Brightwell and Winkler~\cite{BW1} proved \SP-completeness of \ts \LE \ts in~1991, 
resolving an open problem from 1984 (cf.~\ref{ss:finrem-quotes}).  They applied the result
to show that computing the volume of convex polytopes in $\rr^n$ is \SP-hard.  
This connection was first established by Khachiyan \cite{Kha} based on the work of Stanley~\cite{S1}.  
It was later used to improve sorting under partial information, see~\cite{KL}.

The \LE problem is somewhat related to the problem counting order ideals in a poset, 
known to be \SP-complete~\cite{PB}.   In contrast with the
latter problem, \LE has FPRAS which allows $(1+\ve)$-approximation of $\ext(P)$.
This was proved by Karzanov and Khachiyan \cite{KK} and independently by Matthews~\cite{Mat}.
See also~\cite{BD,BGHP,FW,Hu1} for improvement upon the Karzanov--Khachiyan Markov
chain and mixing time bounds.

Moreover, there are several classes of posets for which the counting is known to be polynomial:
the dimension-2 posets given by
Young diagrams of skew shape (see e.g.~\cite{MPP,S2}), the \emph{series-parallel posets}
(also dimension~2, see~\cite[$\S$2.4]{Moh}), a larger class of posets with \emph{bounded
decomposition diameter}~\cite[$\S$4.2]{Moh}, \emph{sparse posets}~\cite{EGKO,KHNK},
posets whose covering graphs have disjoint cycles (see~\cite{Atk}),
and \emph{N-free posets} with bounded width and spread~\cite{FM}.


The \emph{height-2 posets} is an important and well studied class of posets.
Brightwell and Winkler write: ``We strongly suspect that Linear Extension Count for posets
of height~2 is still \SP-complete, but it seems that an entirely different construction is
required to prove this''~\cite{BW1}.  They got this half-right -- note that our construction
builds on top of their result.

In fact, the linear extensions of height-2 posets do seem to have a much easier structure
than the general posets.  For example, Trotter, Gehrlein and Fishburn~\cite{TGF} prove the
famous \emph{$1/3\ts$--$\ts 2/3$ conjecture} for this class, a problem that remains open in full generality.
Similarly, in a recent paper~\cite{CRS}, Caracciolo, Rinaldi and Sportiello, study a 
new Markov chain on linear extensions of height-2 posets, which they call \emph{corrugated surfaces}.
They are motivated by the \emph{Bead Model} in Statistical Mechanics and \emph{standard Young tableau}
sampling.  They claim, based on computer experiments, a nearly linear mixing time for this Markov chain.
Recently, Huber~\cite{Hu2} noticed the connection and proved the nearly linear mixing time
for a different Markov chain in this case.


\emph{Incidence posets} are not as classical as height-2 posets, but have also been studied
quite intensely.  We refer to recent papers~\cite{LS,TW} for an overview of the area and
further references.


The study of posets of a given dimension is an important area, and the dimension~2 is both
the first interesting dimension and special due to the duality property.  See monograph~\cite{Tro}
for a comprehensive treatment.  Posets of dimension~2 have a sufficiently rigid combinatorial
structure to make various computational problems tractable.  For example, the decision problem 
whether a poset has dimension~2 is in~\P (see e.g.~\cite{T2}, as is the above mentioned
problem of counting ideals of dimension-2 posets, see~\cite[p.~163]{Moh}.  Another surprising 
property of dimension-2 posets is an asymptotically sharp lower and upper bounds on the product 
\ts $\ext(P)\ts\ext(P^\ast)$, where $P^\ast$ is the \emph{dual poset} defined above, see~\cite{BBS}.


The \emph{weak Bruhat order} is a fundamental object in Algebraic Combinatorics, well studied
in much greater generality, see e.g.~\cite{Bre}.  As we mentioned above, the connection
between \dtwoposet and \br has been rediscovered a number of times in varying degree of
generality, see~\cite{BW2,FW,Reu}.


Finally, \emph{computer assisted proofs} are relatively rare in computational complexity.
Let us mention~\cite{KKMPS,MR} for two recent \NP-completeness results with
substantial computational component, and~\cite{BDGJ,Zwi} for two older computer assisted
complexity results.  To the best of our knowledge this paper is the only
computer assisted proof of \ts \SP-completeness, and the only one which uses algebraic
systems to encode logical gates.
We refer to~\cite{Mac} for a historical and sociological overview of the method and
further references.

\medskip

\subsection*{Paper structure}
We start with a highly technical proof
of theorems~\ref{BruhatTheorem} and~\ref{Dim2Cor}.  In sections~\ref{s:setup}
and~\ref{s:circuit} we present the construction, in Section~\ref{LemmaProofSection}
we gives proofs of technical lemmas, and in the appendix list systems of algebraic
equations defining parameters of the logical gates. In sections~\ref{HeightTwoPoset}
and~\ref{s:inc-posets} we give complete proofs of theorems~\ref{HeightTwoTheorem}
and~\ref{IncidenceTheorem}, respectively.  Let us emphasize that the 
proof of Theorem~\ref{HeightTwoTheorem} is completely independent from 
the rest of the paper and is streamlined as much as possible to be accessible 
to a larger audience. We conclude with final remarks and open problems 
in Section~\ref{s:fin-rem}.

\bigskip

\section{Basic definitions and notation}
\label{s:notation}

\subsection{Posets}
We assume the reader is familiar with basic definitions on posets, see e.g.~\cite{T2} 
and \cite[Ch.~3]{S2}. We describe a \emph{linear extension} of a poset 
$\cp = (X, <)$ on a set~$X$ with~$n$ elements as an \emph{assignment} of the 
values~$\{1,2,\dots,n\}$ to~$X$, or as a \emph{labeling} of~$X$ by the 
values~$\{1,2,\dots,n\}$. 

Let~$\ell: X \to \{1,2,\dots,n\}$ be a linear extension of~$\cp$, and let~$X$ be given a default
ordering, say~$X = \{x_1,\dots,x_n\}$. Then the function~$i \mapsto \ell(x_i)$ is a permutation in~$S_n$.
We call this the permutation \emph{induced} by the linear extension.

\subsection{Permutations}
For the technical constructions in Section~\ref{s:circuit} we express all permutations in one-line notation,
in other words as a sequence where the integers from~$1$ to~$n$ occur exactly once. For several of these constructions,
we wish to generalize permutations by either omitting or repeating numbers. We can treat an
arbitrary sequence of $n$ integers as a permutation in $S_n$ by relabeling the elements
from $1$ to $n$, from smallest to largest, and, when a number is repeated, from left to right.
For example, we would relabel the sequence
$$
(7,7,5,3,3,5)
$$
by replacing the two $3$'s with a $1$ and a $2$, the two $5$'s with a $3$ and a $4$,
and the two $7$'s with a $5$ and $6$, giving the permutation
$$
(5,6,3,1,2,4).
$$
We will describe this relabeling explicitly where it helps to clarify the presentation, and talk about \emph{shifting} elements up or down.

We use the term \emph{block} exclusively to refer to a sequence of consecutive integers in consecutive position, and write it by replacing the sequence with an integer encased in a box:~$\boxed{3}$.

\subsection{Other notation}

We write $\N = \{0,1,2,\ldots\}$ for the set of nonnegative integers, and 
$\fq$ to denote the finite field with $q$ elements. Let $[n]=\{1,2,\ldots,n\}$   
and $\binom{[n]}{k}$ to denote $k$-subsets of~$[n]$.  
To make our notation more readable, when writing vectors in $\mathbb{F}_q^d$, 
we omit parentheses and commas, so that $(0,1)$ becomes $01$.

We refer to~\cite{MM,Pap} for notation, basic definitions and results 
in computational complexity.  We use $\phi$ for logical gates.  
We introduce a new notation $\phi \rtimes (v_1,v_2)$ to be a result of 
a certain operation corresponding to $(v_1,v_2)$ applied to~$\phi$, see~$\S$\ref{BruhatLogicGates}.  


\bigskip
\section{The setup for \dtwoposet and \br} \label{s:setup}

\subsection{Linear extensions and the Bruhat order} \label{ss:setup-LE}

We begin with a known result that \dtwoposet is equivalent to~\br.

\begin{lemma}[\cite{FW}] 
For every~$\sigma \in S_n$, there exists a poset~$P_{\sigma}$ of dimension two with~$n$ elements such that~$\ext(P_{\sigma}) = \ext(\sigma)$. Conversely, for every poset~$P$ of dimension two with~$n$ elements, there exists~$\sigma \in S_n$ such that~$\ext(P)=\ext(\sigma)$.
\end{lemma}
\begin{proof}
Given a permutation~$\sigma \in S_n$, we form a poset~$P_{\sigma}$ of dimension~$2$ by taking the points~$p_i=(i,\sigma^{-1}(i)) \in \mathbb{R}^2$, with the standard product ordering. A linear extension of~$P_{\sigma}$ is a function from the~$p_i$'s to~$\{1,2,\dots,n\}$, which induces a permutation~$\tau$ as described in Section~\ref{s:notation}. Then~$\tau$ is a linear extension of~$P_{\sigma}$ if and only if~$\tau(i) < \tau(j)$ whenever~$i < j$ and~$\sigma^{-1}(i) < \sigma^{-1}(j)$. When this holds, for~$\omega = \tau^{-1}$ we have~$\omega \leq \sigma$ in the weak Bruhat order, so that $\ext(P_{\sigma})=\ext(\sigma)$.

Conversely, given a poset~$P$ of dimension two, it can be represented as a collection of points~$p_i \in \mathbb{R}^2$ with the product ordering. We translate the points of~$p_i$ so that they are all in the first quadrant, and then, for some sufficiently small~$\varepsilon > 0$, perform the affine transformation:
$$
p_i \mapsto \begin{pmatrix}1&\varepsilon\\ \varepsilon&1\end{pmatrix} p_i.
$$
This transformation ensures that no two points are in the same row or column without changing the ordering on~$P$. Label the points from~$1$ to~$n$, reading from left to right, and replace the~$x$-coordinates with these labels. Similarly, replace the~$y$-coordinates with the labels~$1$ through~$n$, read from bottom to top. The points now represent the poset~$P_{\sigma}$, for some~$\sigma \in S_n$. We thus have $\ext(P)=\ext(P_{\sigma})=\ext(\sigma)$.
\end{proof}

\subsection{Rigid circuits} \label{RigidCircuits}

In this subsection, we define rigid circuits, which will be the principal gadget in our proof of Theorem~\ref{BruhatTheorem}. Visually, a rigid circuit consists of a collection of wires laid out in the plane. The wires run horizontally, from left to right. They carry a binary signal, with a~$1$ representing \textsc{true}, and a~$0$ representing \textsc{false}. Adjacent wires can feed into logic gates, where they interact in some way; wires cannot cross except at logic gates.

At the far left of the picture, the wires represent binary inputs. The bottom wire at the far right is the output wire. The circuit is satisfied by a choice of inputs if the output wire reads \textsc{true}.  Formally, we give the following definitions:

A \emph{circuit state} with~$k$ wires is a vector~$v \in \mathbb{F}_2^k$. A \emph{general rigid circuit} with~$m$ circuit states and~$k$ wires is a sequence of~$m$ circuit states~$(v_1,\dots,v_m)$, each with~$k$ wires, together with a list of relations~$(L_1,\dots,L_{m-1})$ on~$\mathbb{F}_2^k$, such that~$(v_i,v_{i+1}) \in L_i$, for~$1 \leq i \leq m-1$. The relations~$L_i$ we call \emph{logic gates}.

We next define \emph{specialized rigid circuits}, which are the circuits we will use throughout the paper, by restricting our choice of logic gates. We define four \emph{simple logic gates} as follows.

\smallskip

\quad \textsc{Identity} gate~$L_1$: The identity function from~$\mathbb{F}_2 \to \mathbb{F}_2$.

\quad \textsc{Swap} gate~$L_2$: A function from~$\mathbb{F}_2^2 \to \mathbb{F}_2^2$ that sends~$ab \to ba$, for~$a,b \in \mathbb{F}_2$.

\quad \textsc{AndOr} gate~$L_3$: A function from~$\mathbb{F}_2^2 \to \mathbb{F}_2^2$ that sends~$ab \to (a$ \textsc{and} ~$b)(a$ \textsc{or} ~$b$),

\quad \hskip3.0cm where \textsc{and} and \textsc{or} are bitwise operations, for~$a,b \in \mathbb{F}_2$.

\quad \textsc{TestEq} gate~$L_4$: A relation on~$\mathbb{F}_2^2$ that contains~$\bigl\{(11,11),(00,00)\bigr\}$.

\smallskip

\nin Note that the \textsc{TestEq}~gate merely copies the signal when both wires share the same truth value. If the wires contain different truth values, there is no acceptable next circuit state. In this case, we say the circuit \textit{shorts out}.

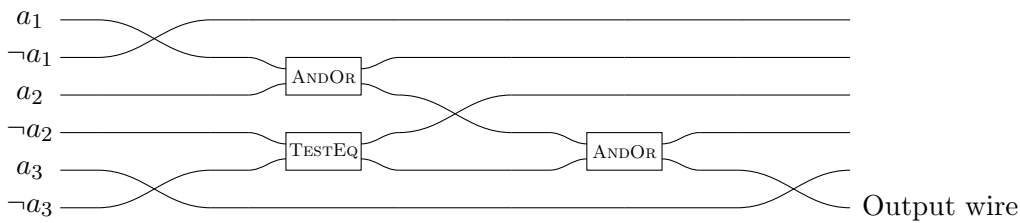
\begin{figure}\label{f:spec}
\centering
\begin{tikzpicture}
\node at (-0.4,0) {$\neg a_3$};
\node at (-0.4,0.5) {$ a_3$};
\node at (-0.4,1) {$\neg a_2$};
\node at (-0.4,1.5) {$ a_2$};
\node at (-0.4,2) {$\neg a_1$};
\node at (-0.4,2.5) {$ a_1$};

\node at (11.7,0){Output wire};

\node (AndOr1) at (7.5,0.75) {\tiny{\textsc{AndOr}}};
\draw (7.0,0.5) rectangle (8.0,1.0);

\node (AndOr2) at (3.5,1.75) {\tiny{\textsc{AndOr}}};
\draw (3.0,1.5) rectangle (4.0,2.0);

\node (TestEq1) at (3.5,0.75) {\tiny{\textsc{TestEq}}};
\draw (3.0,0.5) rectangle (4.0,1.0);

  \draw (0,0) to (0.5,0);
	\draw (0.5,0) to [out=0,in=180](2,.5);
	\draw (2,0) to (2.5,0);
	\draw (2.5,0) to (4.5,0);
	\draw (4.5,0) to (6,0);
	\draw (6,0) to (7.5,0);
	\draw (7.5,0) to (9,0);
	\draw (9,0) to [out=0,in=180] (10.5,0.5);
	
	\draw (0,0.5) to (0.5,0.5);
	\draw (0.5,0.5) to [out=0,in=180](2,0);
	\draw (2,0.5) to (2.5,0.5);
	\draw (2.5,0.5) to [out=0,in=180] (3,0.65);
	\draw (4,0.65) to [out=0,in=180] (4.5,0.5);
	\draw (4.5,0.5) to (6,0.5);
	\draw (6,0.5) to (6.5,0.5);
	\draw (6.5,0.5) to [out=0,in=180] (7,0.65);
	\draw (8,0.65) to [out=0,in=180] (8.5,0.5);
	\draw (8.5,0.5) to (9.0,0.5);
	\draw (9.0,0.5) to [out=0,in=180] (10.5,0);
	
	\draw (0,1) to (0.5,1);
	\draw (0.5,1) to (2,1);
	\draw (2,1) to (2.5,1);
	\draw (2.5,1) to [out=0,in=180] (3,0.85);
	\draw (4,0.85) to [out=0,in=180] (4.5,1);
	\draw (4.5,1) to [out=0,in=180](6,1.5);
	\draw (6,1) to (6.5,1);
	\draw (6.5,1) to [out=0,in=180] (7,0.85);
	\draw (8,0.85) to [out=0,in=180] (8.5,1);
	\draw (8.5,1) to (9.0,1);
	\draw (9,1) to (10.5,1);
	
	\draw (0,1.5) to (0.5,1.5);
	\draw (0.5,1.5) to (2,1.5);
	\draw (2,1.5) to (2.5,1.5);
	\draw (2.5,1.5) to [out=0,in=180] (3,1.65);
	\draw (4,1.65) to [out=0,in=180] (4.5,1.5);
	\draw (4.5,1.5) to [out=0,in=180](6,1);
	\draw (6,1.5) to (7.5,1.5);
	\draw (7.5,1.5) to (9,1.5);
	\draw (9,1.5) to (10.5,1.5);
	
	\draw (0,2.0) to (0.5,2.0);
	\draw (0.5,2) to [out=0,in=180](2,2.5);
	\draw (2,2) to (2.5,2);
		\draw (2.5,2) to [out=0,in=180] (3,1.85);
	\draw (4,1.85) to [out=0,in=180] (4.5,2);
	\draw (4.5,2.0) to (6,2.0);
	\draw (6,2.0) to (7.5,2.0);
	\draw (7.5,2.0) to (9,2.0);
	\draw (9,2) to (10.5,2);
	
	\draw (0,2.5) to (0.5,2.5);
	\draw (0.5,2.5) to [out=0,in=180](2,2);
	\draw (2,2.5) to (2.5,2.5);
	\draw (2.5,2.5) to (4.5,2.5);
	\draw (4.5,2.5) to (6,2.5);
	\draw (6,2.5) to (7.5,2.5);
	\draw (7.5,2.5) to (9,2.5);
	\draw (9,2.5) to (10.5,2.5);
	
\end{tikzpicture}
\caption{A specialized rigid circuit~$C$ with~$\ext(C)=4$. We force~$\neg a_2 = \neg a_3$, and the output wire carries the value of the clause~$(a_1 \vee a_2 \vee \neg a_2)$.}
\end{figure}

Let $v$ and $v'$ be circuit states with $k$ and $k'$ wires, respectively. We define the \emph{coupling} of $v$ and $v'$, which we write as $v \wedge v'$, by concatenating the entries of $v$ and $v'$ to give a circuit state with $k+k'$ wires. Let~$L$ and~$L'$ be logic gates on~$k$ and~$k'$ wires, respectively. We define the \emph{coupling} of~$L$ and~$L'$, which we write as $L \wedge L'$, to be the relation on~$\mathbb{F}_2^{k+k'}$ where $(v_1 \wedge v_1', v_2 \wedge v_2') \in L \wedge L'$ precisely when $(v_1,v_2) \in L$ and $(v_1',v_2') \in L'$. A \emph{compound logic gate} is a logic gate made by coupling together copies of the four simple logic gates.

Note that a compound logic gate on~$(v_i,v_{i+1})$ determines~$v_{i+1}$ from~$v_i$ as long as the circuit does not short out. In our construction, it is sufficient to use compound logic gates where all but one of the gates coupled together are \textsc{Identity}~gates. By abuse of notation, we still generally refer to such a compound logic gate by the one simple gate in the coupling that is not an \textsc{Identity}~gate. So, for example, a compound logic gate that swaps the wires in positions $i$ and $i+1$ and otherwise is made up of \textsc{Identity}~gates we will call a \textsc{Swap}~gate. 

A \emph{specialized rigid circuit} is a general rigid circuit with~$m$ circuit states and~$2k$ wires, such that each logic gate is a compound logic gate and the initial circuit state $v_1 = (a_1,\dots,a_{2k})$ has exactly one of each pair~$a_{2i-1},a_{2i}$ set to \textsc{true}. We therefore relabel the entries of~$v_1$ as~$(a_1,\neg a_1,\dots,a_k,\neg a_k)$, where~$\neg$ denotes bitwise \textsc{not}. A \emph{satisfying assignment} of a circuit $C$ is a choice of~$v_1$ such that the circuit does not short out and the last term of~$v_m$ is set to \textsc{true}.

We refer to circuits by the capital letter~$C$, and call the number of satisfying assignments~$\ext(C)$. We can now state the following:

\smallskip
\nin \hskip1.9cm {\large \rigidcircuit}

\nin \hskip1.9cm \textbf{Input:} A specialized rigid circuit~$C$.

\nin \hskip1.9cm \textbf{Output:} The number~$\ext(C)$ of satisfying assignments of~$C$.
\smallskip

Throughout this paper, we will refer to specialized rigid circuits simply as rigid circuits or as circuits when our meaning is clear. Before moving on, we observe the following:

\begin{lemma}\label{exactlyktrue}
For every rigid circuit~$C$ with a satisfying assignment~$v_1$, there will be exactly~$k$ wires set to \textsc{true} in each circuit state.
\end{lemma}
\begin{proof}
There must be~$k$ wires set to \textsc{true} in~$v_1$. We note that none of the four simple gates can change the number of \textsc{true} wires, which completes the proof.
\end{proof}

Now we give the first step of our reduction from \textsc{\#3SAT}.

\begin{lemma}\label{CircuitTo3Sat}
There is a parsimonious reduction from \textsc{\#3SAT} to \textsc{\#RigidCircuit}.
\end{lemma}
\begin{proof}
Let~$I$ be an instance of \textsc{\#3SAT} with~$u$ variables and~$v$ clauses. We form a rigid circuit with~$6v$ wires, so that there is one pair of wires for each time a variable or its negation appears in a clause.

We label these~$6v$ wires as~$(a_1, \neg a_1, a_2, \neg a_2, \dots, a_{3v}, \neg a_{3v})$. We want some of these wires to represent multiple instances of the same variable. To force~$a_i = a_j$, use \textsc{Swap} gates to move~$a_i$ and~$a_j$ next to each other, and then run them through a \textsc{TestEq}~gate. The circuit will then short out unless both~$a_i = a_j$ and~$\neg a_i = \neg a_j$.

We then use \textsc{Swap}~gates to re-arrange the variables so that the order of variables in the first $3v$ wires match the clauses of $I$. We use two \textsc{AndOr}~gates on each clause to produce the desired disjunctions. At this point in the construction, the $3i$-th wire carries the value of the $i$-th clause, for $i$ between $1$ and $k$. Now, we use more \textsc{Swap}~gates to move these~$k$~wires to the far left of the circuit state, and use $(k-1)$ \textsc{AndOr}~gates to compute the conjunction of all of the clauses, which ends up in the first wire. Finally, we swap the first wire into the last position of our circuit state.

It takes~$O(v)$ uses of \textsc{Swap}~gates to move any two wires adjacent to each other, so this entire process requires~$m = O(v^2)$ circuit states.
\end{proof}

\subsection{Primes and circuits in the Bruhat order}\label{PrimesSubsection}

There is no parsimonious reduction from \textsc{\#RigidCircuit} to \br, because there exist rigid circuits with no satisfying assignments, but every element~$\sigma \in S_n$ has~$\ext(\sigma) \geq 1$. Instead, we will use a collection of permutations~$\sigma$ that allow us to compute the residue of \textsc{\#RigidCircuit} modulo enough primes that we can then use the Chinese Remainder Theorem to compute \textsc{\#RigidCircuit}.

We need the following number theory result:

\begin{prop}[see e.g. \protect{\cite[p.~4]{BW1}}]
\label{PrimesProp}
For~$k \geq 4$, the product of primes between~$k$ and~$k^2$ is at least~$2^kk!$.
\end{prop}

In Section~\ref{s:circuit} we will prove the following:

\begin{mainlem}\label{BruhatToCircuit}
For every rigid circuit~$C$ with~$m$ circuit states and~$2k$ wires,~$k > 7$, 
and every prime~$p$ between~$k$ and~$k^2$, there is \ts $n=O(m\ts k^{10})$, and \ts $\sigma \in S_n$, such that \ts 
$\ext(C) \equiv -\ext(\sigma) \mod{p}$.
\end{mainlem}

\begin{proof}[Proof of Theorem~\ref{BruhatTheorem}.]
We construct a polynomial time reduction from \textsc{\#3SAT} to \br. Given a problem in \textsc{\#3SAT}, we first apply Lemma~\ref{CircuitTo3Sat} to obtain a rigid circuit~$C$ with~$m$ circuit states and~$2k$ wires. We next apply Lemma~\ref{BruhatToCircuit} to find, for each prime~$p$ between~$k$ and~$k^2$, some choice of~$n$ and~$\sigma \in S_n$ with~$\ext(C) \equiv -\ext(\sigma) \mod{p}$. Then, as in \cite{BW1}, we use the Chinese Remainder Theorem to compute the residue of~$\ext(C)$ modulo the product of primes between~$k$ and~$k^2$.

Since there are at most~$2^k$ satisfying assignments of a particular rigid circuit, applying Proposition~\ref{PrimesProp} completes the proof.
\end{proof}
\bigskip

\section{Circuit constructions} \label{s:circuit}

\subsection{Bruhat circuits}
\label{BruhatCircuitSubsection}

To prove Lemma~\ref{BruhatToCircuit}, we produce a permutation~$\sigma$ that emulates the design in \S\ref{RigidCircuits}. We build a Bruhat circuit, with Bruhat circuit states, simple Bruhat logic gates, and compound Bruhat logic gates.

We need to modify our circuits as follows. Let $C$ be a specialized rigid circuit with $2k$ wires and $m$ circuit states. First, we add~$(p-k)$ pairs of wires and use \textsc{TestEq}~gates to set the value of $a_{k+1},a_{k+2},\dots,a_{p}$ equal to the initial value of~$a_1$. We then stack~$(p-1)$ copies of this modified circuit together, and use \textsc{TestEq}~gates to ensure that each copy of the circuit will have the same initial assignment as every other copy.

Next, whenever a \textsc{TestEq} or \textsc{AndOr}~gate acts on a pair of wires (in any copy of the circuit), we use \textsc{Swap}~gates to move those wires to the first two positions of the circuit state vector. We perform the desired \textsc{TestEq} or \textsc{AndOr} operation, and then use \textsc{Swap}~gates to put the wires back in their previous positions.

Finally, we use \textsc{Swap} gates to bring the last wire of each circuit copy into the final~$(p-1)$ positions of our final circuit state. We write $C_p$ for the resulting circuit, and call it a \emph{mod-$p$ parallel circuit}.

The motivation for these modifications comes later, in the technical requirements of Lemma~\ref{TechnicalLemma} and the constructions in \S\ref{InitTest} and \S\ref{modpmod}. For now, though, we note that~$\ext(C) = \ext(C_p)$, and, by Lemma~\ref{exactlyktrue}, in every valid circuit assignment, each circuit state of~$C_p$ will contain exactly~$(p^2-p)$ wires set to \textsc{true}.

A \emph{Bruhat circuit framework} is a permutation~$\sigma \in S_n$ together with a classification of the elements in~$\{1,2,\dots,n\}$ into one of three categories.

The \emph{separators} are a list of elements~$s_1 < s_2 < \cdots < s_m$ with~$\sigma^{-1}(s_1) < \cdots < \sigma^{-1}(s_m)$. By convention we let~$s_0=\sigma^{-1}(s_0)=0$ and~$s_{m+1}=\sigma^{-1}(s_{m+1}) = n+1$ where needed. For each remaining element $x$, there is some $i$, with $0 \leq i \leq m$, such that
$$
\sigma^{-1}(s_i) < \sigma^{-1}(x) < \sigma^{-1}(s_{i+1}).
$$
We require either
$$
s_i < x < s_{i+1},
$$
in which case we call $x$ a \emph{stable element}, or
$$
s_{i-1} < x < s_{i},
$$
in which case we call $x$ a \emph{variable}.

We require that for each choice of~$i$, with~$1 \leq i \leq m$ there are
$$N=2p^2-2p$$
variables. We label the variables satisfying~$\sigma^{-1}(s_i) < \sigma^{-1}(x) < \sigma^{-1}(s_{i+1})$ as~$x_{i1} > x_{i2} > \cdots > x_{iN}$. We require further that~$\sigma^{-1}(x_{i1}) < \sigma^{-1}(x_{i2}) < \cdots < \sigma^{-1}(x_{iN})$.

We now make the following essential observations. Let~$\tau\in S_n$ be chosen with~$\tau \leq \sigma$. Let~$x$ be a stable element and~$x_{ij}$ be a variable satisfying
$$
\sigma^{-1}(s_i) < \sigma^{-1}(x),\sigma^{-1}(x_{ij}) < \sigma^{-1}(s_{i+1}).$$ Then:
$$
\tau^{-1}(s_1) < \cdots < \tau^{-1}(s_m) \qquad \text{and}\qquad \tau^{-1}(s_i) < \tau^{-1}(x) < \tau^{-1}(s_{i+1}),
$$
and either
$$
\tau^{-1}(s_i) < \tau^{-1}(x_{ij}) < \tau^{-1}(s_{i+1}) \qquad \text{or} \qquad \tau^{-1}(s_{i-1}) < \tau^{-1}(x_{ij}) < \tau^{-1}(s_i).
$$
Given a Bruhat circuit framework~$\sigma$ and some~$\tau \leq \sigma$, for~$1 \leq i \leq m$ we assign to~$\tau$ a \emph{Bruhat circuit state}~$v_i \in \mathbb{F}_2^{N}$ as follows. Write~$v_i=(a_{i1},a_{i2},\dots,a_{iN})$, with~$a_{ij} \in \mathbb{F}_2$. Then take~$a_{ij} = 1$ if~$\tau^{-1}(s_i) < \tau^{-1}(x_{ij}) < \tau^{-1}(s_{i+1})$, and~$a_{ij} = 0$ otherwise.

Note that the since the~$y_{ij}$'s are arranged in strictly decreasing order, they can be re-arranged arbitrarily in~$\tau$, so that every possible circuit state can be realized as a Bruhat circuit state. In particular, for every possible circuit assignment~$(v_1,\dots,v_m)$, there is a unique permutation $\tau$ with Bruhat circuit state equal to~$(v_1,\dots,v_m)$ maximal in the Bruhat order. It is obtained by moving each variable which takes the value~\textsc{false} in circuit state~$v_i$ immediately to the left of~$s_i$, keeping those \textsc{false} variables in descending order. Call this permutation~$\tau | v_1,\dots,v_m$.

In summary, we have:
\begin{equation}\label{lazysum}
\ext(\sigma) \, = \, \sum_{(v_1,\dots,v_m)} \ext(\tau | v_1,\dots,v_m)\ts,
\end{equation}
where the sum is taken over every possible set of circuit states~$(v_1,\dots,v_m)$. In the next four subsections, we will show how to control the value of~$\ext(\tau | v_1,\dots,v_m)$ to encode the logic of our circuit.

\subsection{Bruhat logic gates} \label{BruhatLogicGates}
A \emph{Bruhat logic gate} with~$k$ wires is a sequence~$\phi$ of distinct integers such that the smallest~$k$ terms are in decreasing order, the last~$k$ terms are in decreasing order, and these two sets do not overlap. We refer to these elements as the input and output variables, respectively, of the logic gate.

For technical reasons, we also require that immediately preceding the last~$k$ terms is a block of~$(p^3-1)$ consecutive elements, all less than the last~$k$ terms. We refer to this block, appropriately, as the \emph{penultimate block}.

If~$|\phi| = \ell$, we do not require~$\phi$ to take values strictly in the set~$\{1,\dots,\ell\}$, but we still treat~$\phi$ as a member of~$S_{\ell}$, as described in Section~\ref{s:notation}

The \emph{evaluation} of a Bruhat logic gate~$\phi$ with~$k$ wires at some pair of circuit states~$(v_1,v_2) \in \mathbb{F}_2^k$ is given by deleting from~$\phi$ each of the input variables corresponding to a~$0$ in~$v_1$ and each of the output variables corresponding to a~$1$ in~$v_2$. We write this as~$\phi\rtimes(v_1,v_2)$.

Given a Bruhat circuit framework~$\sigma$, we write down a collection of sequences~$(\sigma_1,\dots,\sigma_{m+1})$ as follows. For~$\sigma_i$, write down all the elements of~$\sigma$ (taken in one line notation) between~$s_{i-1}$ and~$s_i$, and then write down only the variables that occur between~$s_i$ and~$s_{i+1}$.

Note that, for all~$i$ satisfying~$2 \leq i \leq m$, the sequence~$\sigma_i$ is a Bruhat logic gate with~$N=2p^2-2p$ wires. Also note that the choice of~$(\sigma_1,\dots,\sigma_{m+1})$ determines our original Bruhat circuit framework~$\sigma$ uniquely.

For a given set of circuit states~$(v_1,\dots,v_m)$, we similarly define the sequence $(\tau_1,\dots,\tau_{m+1})$, by writing~$\tau | v_1,\dots,v_m$ in one-line notation and breaking it apart at each separator~$s_i$. Note that~$\tau_i = \sigma_i\rtimes(v_{i-1},v_i)$, for~$2 \leq i \leq m$. By abuse of notation, we set~$v_0=v_{m+1}=\varnothing$, and let~$\sigma_1\rtimes(v_0,v_1)=\tau_1$ and~$\sigma_{m+1}\rtimes(v_m,v_{m+1}) = \tau_{m+1}$.

Though the sequences~$\tau_i$ are not permutations, we can treat them as permutations as described in Section~\ref{s:notation}, and so compute~$\ext(\tau_i)$. We can now rewrite \eqref{lazysum} as
\begin{equation}\label{bettersum}
\ext(\sigma) = \sum_{(v_1,\dots,v_m)} \prod_{i=1}^{m+1} \ext(\tau_i)\ = \sum_{(v_1,\dots,v_m)} \prod_{i=1}^{m+1} \ext\bigl(\sigma_i\rtimes(v_{i-1},v_i)\bigr),
\end{equation}
where the sum is taken over every possible set of circuit states~$(v_1,\dots,v_m)$.

We must have this product take the value~$0$ modulo~$p$ whenever~$(v_1,\dots,v_m)$ is not a satisfying assignment of~$C_p$, and to take some nonzero constant value otherwise.

The simplest way to do this would be to construct Bruhat logic gates~$\sigma_i$ so that
$$
\ext\bigl(\sigma_i\rtimes(v_{i-1},v_i)\bigr) \equiv 0 \mod{p}\.,\quad \text{when~$(v_{i-1},v_i) \not \in L_{i-1}$,}
$$
and
$$
\ext\bigl(\sigma_i\rtimes(v_{i-1},v_i)\bigr) \equiv 1 \mod{p}\/,\quad \text{when $(v_{i-1},v_i) \in L_{i-1}$.}
$$
To keep computations manageable, we weaken the condition by sometimes only requiring
$$
\ext\bigl(\sigma_i\rtimes(v_{i-1},v_i)\bigr) \not \equiv 0 \mod{p}\.,\quad\text{when~$(v_{i-1},v_i) \in L_{i-1}$.}
$$
In \S\ref{modpmod}, we explain how to complete the proof of Main Lemma \ref{BruhatToCircuit} under these weakened conditions. We postpone the construction of~$\sigma_1$ and~$\sigma_{m+1}$ until \S\ref{InitTest}.

As with rigid circuits, we say that a Bruhat circuit~$\phi$ \textit{shorts out} at~$v$ if, for every choice of~$v'$, we have
$$\ext\bigl(\phi\rtimes(v,v')\bigr) \equiv 0 \mod{p}.$$

\subsection{Bruhat compound logic gates}


In this subsection, we explain how to couple two Bruhat logic gates~$\phi$ with~$k$ wires and~$\phi'$ with~$k'$ wires to produce a new Bruhat logic gate~$\phi \wedge \phi'$ with~$k+k'$ wires, emulating the behavior of coupled logic gates defined in \S\ref{RigidCircuits}.

First, we give a technical construction to ensure that the number of wires carrying the value~\textsc{true} remains constant through logic gates. This matches the statement for rigid circuits given in Lemma~\ref{exactlyktrue}. Given a Bruhat logic gate~$\phi$ with~$k$ wires, we say~$\phi$ is \emph{balanced} modulo~$p$ if~$|\phi|-k \equiv 0 \mod{p^3}$.

To construct the \emph{restriction} of~$\phi$, which we write~$\phi_\circ$, we append to the beginning of~$\phi$ the block~$(\max(\phi)+1,\max(\phi)+2,\dots,\max(\phi)+p^3-1)$, which we call the \emph{initial block} of~$\phi_{\circ}$. We have:

\begin{lemma} \label{LimitingVariablesPassed}
For a Bruhat logic gate~$\phi$ that is balanced modulo~$p$, we have
$$
\ext\left(\phi_\circ\rtimes(v,v')\right) \equiv 0 \mod{p}
$$
when ~$v$ and~$v'$ do not have an equal number of wires carrying the value~\textsc{true}. When~$v$ and~$v'$ do have an equal number of wires carrying the value~\textsc{true}, we have
$$
\ext\left(\phi_\circ\rtimes(v,v')\right) \equiv \ext\left(\phi\rtimes(v,v')\right) \mod{p}.
$$
\end{lemma}

The proof of this lemma is given in Subsection~\ref{ProofLimitingVariablesPassed}.
\smallskip

We now return to the construction of~$\phi \wedge \phi'$. We restrict to the case where~$\phi$ is one of our simple Bruhat gates, and where~$\phi'$ is a compound gate made up of \textsc{Identity} and \textsc{Swap} gates, since the construction in \S\ref{BruhatCircuitSubsection} requires us to place every \textsc{TestEq} or \textsc{AndOr} gate at the top of our circuit. We require further that~$\phi$ and~$\phi'$ be balanced modulo~$p$.

The construction of~$\phi \wedge \phi'$ involves inserting~$\phi_{\circ}'$ in place of the penultimate block of~$\phi$, shifting elements appropriately. To explain these shifts, we replace each of the elements of $\phi$ and $\phi_{\circ}'$ with ordered pairs of integers.

Let~$y$ be the value of the first entry of the penultimate block of~$\phi$. Replace each of the elements~$x$ in~$\phi$ with the ordered pair~$(x,0)$. Replace the input variables $x$ of $\phi_{\circ}'$ with $(0,x)$, and all other elements $x$ of $\phi_{\circ}'$ with $(y,x)$. Then delete the penultimate block of~$\phi$ and insert the relabeled $\phi_{\circ}'$ in its place.

Now relabel the entries from $1$ to $|\phi|+|\phi'|$, going in order from smallest to largest with respect to the lexicographical order on $\mathbb{Z}^2$. Call the result $\phi \wedge \phi'$. Note that~$\phi \wedge \phi'$ is a Bruhat logic gate with~$k+1$ wires when~$\phi$ is the \textsc{Identity} gate, and~$k+2$ wires otherwise, and that~$\phi \wedge \phi'$ is balanced modulo~$p$.




We define the following operations on logic gates:

\begin{definition}
\textit{Left insertion}, \textit{middle insertion} and \textit{right insertion}, denoted~$L(\phi)$,~$M(\phi)$ and~$R(\phi)$, respectively, are operators on Bruhat logic gates defined as follows. The terms left, middle and right are all oriented with respect to the penultimate block. Left insertion inserts the element~$1$ into~$\phi$ immediately to the left of the penultimate block, and shifts all other elements up by~$1$. Middle insertion increases the length of the penultimate block by~$1$. Right insertion inserts an element one larger than the largest element in the penultimate block to the very end of~$\phi$, and shifts all larger elements up by~$1$.

Also, let~$M^{-1}(\phi)$ denote the inverse operation to~$M$ where we decrease the length of the penultimate block by~$1$.
\end{definition}
The following lemma gives the set of conditions required for the coupling of logic gates to behave as desired. These equations produce the polynomials given in Appendix~\ref{GateEqs} that are used in \S\ref{modpmod} to complete the proof of Main Lemma~\ref{BruhatToCircuit}.
\begin{lemma} \label{TechnicalLemma}
Given~$\phi$ and~$\phi'$ as above, if~$\phi$ is not the identity gate,~$(\phi \wedge \phi')_{\circ}$ behaves as the coupling of the logic gates associated to~$\phi_{\circ}$ and~$\phi_{\circ}'$ when the following six equations are satisfied:
\begin{align*}
&|\phi|-\ts 2 \. \equiv \. 0 \mod{p^3}, \tag{\textsf{1}} \\
& -\ts 2 \ts \ext\bigl(M(\phi_\circ)\rtimes(10,11)\bigr) \ts + \ts \ext\bigl(L(\phi_\circ)\rtimes(10,11)\bigr) \ts + \ts  \ext\bigl(R(\phi_\circ)\rtimes(10,11)\bigr) \. \equiv \. 0 \mod{p}, \tag{\textsf{2}}\\
& -\ts 2 \ts \ext\bigl(M(\phi_\circ)\rtimes(01,11)\bigr) \ts + \ts  \ext\bigl(L(\phi_\circ)\rtimes(01,11)\bigr) \ts + \ts  \ext\bigl(R(\phi_\circ)\rtimes(01,11)\bigr) \. \equiv \. 0 \mod{p}, \tag {\textsf{3}}\\
& -\ts 2 \ts \ext\bigl(M(\phi_\circ)\rtimes(00,01)\bigr) \ts + \ts  \ext\bigl(L(\phi_\circ)\rtimes(00,01)\bigr) \ts + \ts  \ext\bigl(R(\phi_\circ)\rtimes(00,01)\bigr) \. \equiv \. 0 \mod{p}, \tag{\textsf{4}}\\
& -\ts 2 \ts \ext\bigl(M(\phi_\circ)\rtimes(00,10)\bigr) \ts + \ts  \ext\bigl(L(\phi_\circ)\rtimes(00,10)\bigr) \ts + \ts  \ext\bigl(R(\phi_\circ)\rtimes(00,10)\bigr) \. \equiv \. 0 \mod{p}, \tag{\textsf{5}}\\
& 2\ts \ext\bigl(M^2(\phi_\circ)\rtimes(00,11)\bigr) \ts - \ts  4 \ts \ext\bigl(LM(\phi_\circ)\rtimes(00,11)\bigr) \ts - \ts  4\ts  \ext\bigl(RM(\phi_\circ)\rtimes(00,11)\bigr) \ts + \ts  \\
& \ext\bigl(L^2(\phi_\circ)\rtimes(00,11)\bigr) \ts + \ts  2\ts \ext\bigl(LR(\phi_\circ)\rtimes(00,11)\bigr) \ts + \ts  \ext\bigl(R^2(\phi_\circ)\rtimes(00,11)\bigr) \. \equiv \. 0 \mod{p}. \tag{\textsf{6}}
\end{align*}

When~$\phi$ is the identity gate, we need two equations to be satisfied:
\begin{align*}
&|\phi|-1 \. \equiv \. 0 \mod{p^3}, \tag{\textsf{7}}\\
-2 \ext\bigl(M(\phi_\circ)\rtimes(0,1)\bigr) &+ \ext\bigl(L(\phi_\circ)\rtimes(0,1)\bigr) + \ext\bigl(R(\phi_\circ)\rtimes(0,1)\bigr) \. \equiv \. 0 \mod{p}. \tag{\textsf{8}}\\
\end{align*}
\end{lemma}
The proof of this lemma is given in Subsection~\ref{ProofTechnicalLemma}.

\subsection{Initializing and testing wires}\label{InitTest}

We now give explicitly the construction of~$\sigma_1$ and~$\sigma_{m+1}$, to initialize wires at the beginning of our circuit and test the value of the output wire at the end.

For~$\sigma_1$, we begin with~$\psi$, a compound Bruhat logic gate consisting of~$p^2-p$ copies of the identity wire, and then take~$\sigma_1 = \psi_{\circ} \rtimes(\vec{0},\vec{0})$, where~$\vec{0}$ represents a circuit state with all wires set to~\textsc{false}. The identity gate construction given in Lemma~\ref{Gates} is simple enough that we can state what~$\sigma_1$ looks like explicitly. It contains a sequence of~$p^2-p+1$ blocks, each of size~$p^3-1$, followed by the variables. The blocks themselves decrease by~$p^3$ with each block step, and the variables fill in the missing terms.

For example, when $p=2$, $\sigma_1$ has~$2$ wires, and we have
$$
\sigma_1 = \boxed{17}\ \boxed{9}\ \boxed{1}\ 16\ 8,
$$
where each of the numbers in boxes represent blocks of size~$p^3-1$, using the~$\boxed{x}$ notation as in \S\ref{BruhatCircuitSubsection} above. For ease in notation, we shift elements down and write instead
$$
\sigma_1 = \boxed{5}\ \boxed{3}\ \boxed{1}\ 4\ 2.
$$
We then modify~$\sigma_1$ by duplicating each of the~$p^2-p$ terms at the end, in their respective positions. Our example above becomes
$$\sigma_1 = \boxed{5}\ \boxed{3}\ \boxed{1}\ 4\ 4\ 2\ 2.
$$
Finally, we shift all elements of~$\sigma_1$ up so that all of the elements are distinct and the final sequence is in strictly decreasing order. Our example now reads
$$
\sigma_1 = \boxed{7}\ \boxed{4}\ \boxed{1}\ 6\ 5\ 3\ 2.
$$
Note that $\sigma_1$ now has~$N=2p^2-2p$ output wires, as required.

Recall that, by the abuse of notation introduced in \S\ref{BruhatLogicGates}, for a vector
$$
v = (a_1,\dots,a_{N}) \in \mathbb{F}_2^{N},
$$
we let~$\sigma_1(v_0, v)$ represent~$\sigma_1$ after deleting every element corresponding to a~$1$ in~$v$, where we take~$v_0 = \varnothing$.

Having given the details of our construction, we now prove the following:

\begin{lemma} \label{init}
We have $\sigma_1 \rtimes(v_0,v) \equiv 1 \mod{p}$ precisely when exactly one of each pair~$\{a_{2i-1},a_{2i}\}$ is equal to~$1$, and~$\sigma_1\rtimes(v_0,v)\equiv 0 \mod{p}$ otherwise.
\end{lemma}
\begin{proof}
By construction, when exactly one of each pair~$\{a_{2i-1},a_{2i}\}$ is equal to~$1$, then~$\sigma_1\rtimes(v_0,v) = \psi_{\circ}\rtimes(\vec{0},\vec{0})$, so~$\ext(\sigma_1(v)) = \ext(\psi_{\circ} \rtimes (\vec{0},\vec{0})) \equiv 1 \mod{p}$. And, by Lemma~\ref{LimitingVariablesPassed},~$\ext(\sigma_1\rtimes(v_0,v)) \equiv 0 \mod{p}$ unless~$|v| = p^2-p$.

The only case left to consider is when~$|v| = p^2-p$, but the condition of this lemma is false. In this case, there must be some~$i$ for which both~$a_{2i-1}$ and~$a_{2i}$ are equal to~$1$. However,~$\sigma_1$ will then contain two adjacent blocks of size~$p^3-1$ with no other elements whose values lie between those two blocks. In our example above, deleting both the~$6$ and the~$5$ leaves adjacent blocks~$\boxed{7}$ and~$\boxed{4}$ in~$\sigma_1$.

The rearrangement of these two blocks are therefore independent of the rearrangement of the rest of~$\sigma_1$, so that~$\ext(\sigma_1\rtimes(v_0,v))$ is divisible by
$$
\binom{2p^3-2}{p^3-1}.
$$
This is $p^3C_{p^3-1}$, where $C_{p^3-1}$ is the $p^3-1$-th Catalan number. Thus the binomial coefficient is divisible by~$p$, so~$\ext(\sigma_1\rtimes(v_0,v)) \equiv 0 \mod{p}$, as desired.
\end{proof}

We now give the construction of~$\sigma_{m+1}$. Recall, by the construction given in \S\ref{BruhatCircuitSubsection}, that the final~$p-1$ wires should all carry the value of the output wire, while the other~$N-(p-1)=2p^2-3p+2$ wires are grouped into~$p-1$ sets of~$2p-1$ wires.

We begin with
$$
\sigma_{m+1} = (N,\dots,2,1).
$$
This choice is forced on us, since the input variables always must be in decreasing order and the smallest elements of the permutation. To finish our construction, we insert the sequence
$$\bigl(N+1,N+2,\dots,N+(p-1)\bigr)
$$
into~$\sigma_1$ so as to divide the variables into the sets described above, i.e.\ first~$p-1$ sets of~$2p-1$ wires, followed by a final set of~$p-1$ wires.

We call the sequence~$\bigl(N+1,N+2,\dots,N+(p-1)\bigr)$ \textit{dividers}. Again, by abuse of notation, given a vector~$v \in \mathbb{F}_2^{N}$, we write~$\sigma_{m+1}\rtimes(v,\varnothing) = \sigma_{m+1}\rtimes(v,v_{m+1})$ for the sequence obtained by deleting from~$\sigma_{m+1}$ each variable corresponding to a~$0$ in~$v$.

\begin{lemma} \label{TestingWires}
If the last element of~$v$ is~$1$, i.e.\ if the output wire contains the value~\textsc{true}, then $\ext(\sigma_{m+1}\rtimes (v,v_{m+1})) \equiv -1 \mod{p}$. Otherwise,~$\ext(\sigma_{m+1}\rtimes(v,v_{m+1})) \equiv 0 \mod{p}$.
\end{lemma}
\begin{proof}
By construction, the entire final set of~$(p-1)$ wires will either all be \textsc{true} or all be \textsc{false}. Before all the swapping we did at the end of our circuit in \S\ref{BruhatCircuitSubsection}, each of the original~$p-1$ circuits contained~$2p$ wires, with~$p$ wires set to \textsc{true}. So, after the swapping, if the final set of wires is \textsc{true}, each of the other sets will have~$p-1$ wires set to \textsc{true}, while if the final set of wires is \textsc{false}, each of the other sets will have~$p$ wires set to \textsc{true}.

\smallskip

We treat each case separately:

\smallskip

\nin \textbf{\textit{Case 1}}: The desired wire carries the value~\textsc{false}. Then the~$p$ wires set to \textsc{true} in each of the other~$p-1$ sets correspond in~$\sigma_{m+1}(v,v_{m+1})$ to a sequence of~$p$ consecutive decreasing elements. There are~$p!$ rearrangements of each of those sets, and the rearrangements of those sets are independent of the rearrangement of the rest of the elements in the permutation, so that~$\ext(\sigma_{m+1}(v,v_{m+1})) \equiv 0 \mod{p}$.

\smallskip
\nin \textbf{\textit{Case 2}}: The desired wire carries the value~\textsc{true}. Then the~$p-1$ wires set to \textsc{true} in each of the other~$p-1$ sets correspond in~$\sigma_{m+1}(v,v_{m+1})$ to a sequence of~$p-1$ consecutive decreasing elements. The final set of wires also contains~$p-1$ decreasing elements. We count the number of rearrangements~$\tau \leq \sigma_{m+1}(v,v_{m+1})$ based on the position of the dividers in~$\tau$.

\smallskip
When the dividers remain in exactly the same position, then no other element can move out of its set either, since the variables in~$\tau$ must remain to the left of every divider they were already to the left of in~$\sigma$. This leaves only rearrangements within the~$p$ sets of~$p-1$ strictly decreasing elements, for a total count of \ts $\bigl((p-1)!\bigr)^{p} \equiv (-1)^p \equiv -1 \mod{p}$ \ts by Wilson's theorem.

For all other choices of positions for the dividers, there will be some gap of size at least~$p$ between dividers. The number of ways to rearrange the at least~$p$ elements that fill this gap will be divisible by~$p!$ These rearrangements are independent of the rearrangement of the rest of the sequence, and so the contribution from every other choice of positions for the dividers is~$0$ modulo $p$. Thus the total number of rearrangements is congruent to~$-1 \mod{p}$, as desired.
\end{proof}

\subsection{Parametrized gates}

We now describe how to construct a parametrized family of logic gates whose count of rearrangements is a polynomial in the parameters. We use this construction in \S\ref{modpmod} to give \textsc{Swap}, \textsc{AndOr} and \textsc{TestEq} gates for arbitrary primes~$p$.

Let~$\phi$ be a logic gate containing an increasing sequence~$(x_1,\dots,x_t)$. We form the \emph{parametrization} of~$\phi$ with respect to~$(x_i)$ by replacing each of the elements~$x_i$ with a block of consecutive elements of length~$z_i \in \mathbb{N}$ and shifting the other elements of~$\phi$ up appropriately, and denote this as~$\phixz$.

We require that the sequence of~$x_i$'s conclude prior to the penultimate block of~$\phi$, and by convention write~$x_{t+1}$ for the penultimate block of~$\phi$, with~$z_{t+1} = p^3-1$.

\begin{lemma} \label{parametrization}
For every parametrization~$\phixz$ of a logic gate~$\phi$, there is a polynomial~$g(z_1,\dots,z_{t+1})$ over~$\mathbb{Q}$ such that~$g(z_1,\dots,z_t,p^3-1) = \ext\left(\phixz\right)$ for every~$p$.
\end{lemma}
The proof of this lemma is given in Subsection~\ref{ProofParametrization}.

\subsection{Mod-$p$ modification} \label{modpmod}

Recall that for a logic gate $L$, we say~$(v_1,v_2) \in L$ whenever~$(v_1,v_2)$ satisfies the logic gate, and $(v_1,v_2) \not \in L$ otherwise. For computational reasons, it is easier to find simple logic gates if we relax the condition that~$\ext(\sigma \rtimes(v_1,v_2)) \equiv 1 \mod{p}$ whenever~$(v_1,v_2)$ satisfies the logic gate. We never alter the set of conditions~$\ext(\sigma \rtimes(v_1,v_2)) \equiv 0 \mod{p}$ when~$(v_1,v_2)$ fails to satisfy the logic gate. We describe the modified conditions below.

\smallskip

\textsc{Identity} gate~$L_1$: $\ext\bigl(\sigma \rtimes(v_1,v_2)\bigr) \equiv 1 \mod{p}$ whenever~$(v_1,v_2) \in L_1$.

\textsc{Swap} gate~$L_2$: $\ext\bigl(\sigma \rtimes(v_1,v_2)\bigr) \equiv 1 \mod{p}$ whenever~$(v_1,v_2) \in L_2$.

\textsc{AndOr} gate~$L_3$: $\ext\bigl(\sigma \rtimes(v_1,v_2)\bigr) \not \equiv 0 \mod{p}$ whenever~$(v_1,v_2) \in L_3$.

\textsc{TestEq} gate~$L_4$: $\ext\bigl(\sigma \rtimes(v_1,v_2)\bigr) \not \equiv 0 \mod{p}$ whenever~$(v_1,v_2) \in L_4$.

\smallskip

\nin We now have all of the conditions on our simple Bruhat logic gates, allowing us to state and prove the following:

\begin{lemma} \label{Gates}
For every prime~$p \geq 2$, there is an \textsc{Identity} gate satisfying the condition of Lemma~\ref{TechnicalLemma}. In addition, for each of \textsc{Swap}, \textsc{AndOr} and \textsc{TestEq}, there are parametrized Bruhat logic gates~$\phi$ such that the conditions above together with the conditions in Lemma~\ref{TechnicalLemma} give a system of polynomial equations in the parameters~$\{z_i\}$ that has solutions modulo~$p$ for all primes~$p \geq 11$.
\end{lemma}
\begin{proof}
We prove the statement for each gate by giving an explicit construction. The following permutations represent our four desired gates before restriction. The \textsc{Identity} gate works correctly modulo~$p$ without parametrization. The other three gates are parametrized with respect to the sequence~$\{3,4,5,6,7\}$. The~$\boxed{2}$ in the \textsc{Identity} gate and the~$\boxed{8}$ in the other three gates represent the penultimate blocks of size~$p^3-1$.

We treat each of the equations in Lemma~\ref{TechnicalLemma} first over~$\mathbb{Q}$, so that the equations correspond to some algebraic variety over~$\mathbb{Q}$. For each gate, we are able to give explicitly a rational point on that variety, and the rational nonzero values taken by the \textsc{AndOr} and \textsc{TestEq} gates. For every prime $p \geq 11$, this corresponds to a solution modulo $p$.
\medskip

\begin{enumerate}
\item \nin \textsc{Identity} gate:
$$\phi = 1\ \boxed{2}\ \ 3.$$
The equations that $\phi$ must satisfy are~(\textsf{7}) and~(\textsf{8}) from Lemma~\ref{TechnicalLemma}, and the following four equations:
$$
\ext\bigl(\phi \rtimes(0,0)\bigr) \, = \, \ext\bigl(\phi\rtimes(1,1)\bigr) \. = \. 1\ts, \quad 
\ext\bigl(\phi \rtimes(0,1)\bigr)\, =\, \ext\bigl(\phi \rtimes(1,0)\bigr) \. = \. 0\ts.
$$
The last two equations are guaranteed to be satisfied by Lemma~\ref{LimitingVariablesPassed}. Since~$|\phi| = p^3+1$, we have that~(\textsf{7}) is satisfied. Note that
\begin{align*}
&\ext\bigl(\phi\rtimes(0,0)\bigr) \, = \, \ext\bigl(\phi\rtimes(1,1)\bigr)\, = \, \ext\bigl(L(\phi)\rtimes(0,1)\bigr)\\
&\quad = \, \ext\bigl(M(\phi)\rtimes(0,1)\bigr) \, = \, \ext\bigl(R(\phi)\rtimes(0,1)\bigr) \. =\. 1\ts,
\end{align*}
since these permutations are all just strictly increasing sequences of length~$p^3$. Thus the remaining two equations given here and~$(\textsf{8})$ are satisfied, as desired.
\smallskip

For each of the remaining three gates, there are~$6$ equations from Lemma~\ref{TechnicalLemma}, and an additional~$16$ equations for every possible pair of input and output wires. However, applying Lemma~\ref{LimitingVariablesPassed} as above, we see that~$10$ of these equations will be satisfied automatically. For the remainingg~$12$ equations, we use Lemma~\ref{parametrization} to compute the corresponding polynomials in~$\{z_i\}$, for $1 \leq i \leq 5$. We write out these polynomials explicitly in Appendix~\ref{GateEqs}.
\medskip

\item \nin \textsc{Swap} gate:
$$\phi = 2\ \boxed{3}\ 12\ \boxed{4}\ 1\ \boxed{5}\ 10\ \boxed{6}\ 13\ \boxed{7}\ \boxed{8}\ 11\ 9.$$
The system of equations in \S\ref{swapeqn} has a unique solution over $\mathbb{Q}$: $$(z_1,z_2,z_3,z_4,z_5) = (-1,-2,0,1,-2),$$ so the system of equations is solvable$\mod{p}$, for every prime~$p \geq 2$.
\medskip

\item \nin \textsc{AndOr} gate:$$\phi = 2\ \boxed{3}\ 13\ \boxed{4}\ 11\ \boxed{5}\ 1\ \boxed{6}\ 10\ \boxed{7}\ \boxed{8}\ 12\ 9.$$ The system of equations in \S\ref{andoreqn} reduces to a two-dimensional variety over~$\mathbb{Q}$ of degree~$2$, with infinitely many rational points, including the point
$$(z_1,z_2,z_3,z_4,z_5) = (-2,1,-3,1,-1).$$
The nonzero values~$\ext(\sigma(v,v'))$ takes are~$2$ and~$4$, so we require~$p \neq 2$, and the system of equations is solvable$\mod{p}$ for every prime~$p \geq 3$.
\medskip

\item \nin \textsc{TestEq} gate:$$\phi = 2\ \boxed{3}\ 12\ \boxed{4}\ 10\ \boxed{5}\ 1\ \boxed{6}\ 13\ \boxed{7}\ \boxed{8}\ 11\ 9.$$ The system of equations in \S\ref{TestEqeqn} reduces to a one-dimensional variety over~$\mathbb{Q}$ of degree~$1$, with infinitely many rational points, including the point$$(z_1,z_2,z_3,z_4,z_5) = (-2,-\tfrac{8}{3},\tfrac{5}{3},-3,2),$$ with nonzero values of~$\tfrac{7}{3}$ and~$\tfrac{-8}{3}$ for~$\ext(\sigma(v,v'))$, so that the system of equations is solvable $\mod{p}$ for every prime~$p \geq 11$.
\end{enumerate}
\end{proof}
\subsection{Proof of Main Lemma~\ref{BruhatToCircuit}.}
Given a rigid circuit~$C$, we construct the mod-$p$ parallel circuit~$C_p$ with~$\ext(C)=\ext(C_p)$. We then construct a Bruhat circuit~$\sigma$ that mirrors the behavior of~$C_p$.

By Lemma~\ref{init}, our choices of variable assignments~$v_1$ in the sum in \eqref{bettersum} are restricted to those with~$N=2p^2-2p$ wires grouped in pairs, with exactly one wire set to \textsc{true} and one wire set to \textsc{false} in each pair. By lemmas~\ref{TechnicalLemma} and~\ref{Gates}, the inside product in~\eqref{bettersum} is congruent to~$0 \mod{p}$ except when~$(v_1,\dots,v_m)$ is a set of circuit states satisfying~$C_p$.

By the parallel circuit construction, every time an \textsc{AndOr} gate operation occurs in~$C$, it occurs~$p-1$ times in~$C_p$, acting on the same set of truth values, which gives a contribution of~$1 \mod{p}$ to the product in \eqref{bettersum}.

The same is true for the \textsc{TestEq} operations that occur after the parallel circuit has already been constructed. For the \textsc{TestEq} operations that occur in the construction of the parallel circuit (i.e.\ the \textsc{TestEq} operations used to force each of the copies of the circuit to have the same initial truth values), just repeat them~$p-1$ times, which has no impact on the operation of the circuit.

The \textsc{Identity} and \textsc{Swap} operations all also give a contribution of~$1\mod{p}$ to the product, by construction.

In summary, the contribution to the product from the operation of each of the gates is~$1$. The only contribution left comes from~$\sigma_{m+1}$, which, by Lemma~\ref{TestingWires}, multiplies the product by~$-1$ if the output wire is \textsc{true} and~$0$ otherwise.
\qed
\bigskip

\section{Proof of lemmas} \label{LemmaProofSection}

\subsection{Proof of Lemma \ref{LimitingVariablesPassed}.}
\label{ProofLimitingVariablesPassed}
Write~$|v|$ for the number of wires carrying the value~\textsc{true} in~$v$. Since the elements in the initial block of~$\phi_\circ$ are larger than every other element of~$\phi$, the Bruhat order gives no restriction on the position of these elements relative to the position of the elements of~$\phi$ in a rearrangement~$\tau \leq \phi_\circ$. Thus:
\begin{align*}
\ext\left(\phi_\circ\rtimes(v,v')\right) \, &= \, \binom{|\phi\rtimes(v,v')|+p^3-1}{p^3-1} \, \ext\left(\phi\rtimes(v,v')\right) \\
&= \binom{|\phi| - (k-|v|) - |v'|+p^3-1}{p^3-1} \, \ext\left(\phi\rtimes(v,v')\right).
\end{align*}
Since~$\phi$ is balanced, we know~$|\phi|-k \equiv 0 \mod{p^3}$. Write~$|\phi|-k = ap^3$ and $|v|-|v'| = b$. We have $|b| \leq 2p^2+2p$. Observe that, for integers~$a,b$ with $a>0$ and~$0 < |b| \leq 2p^2+2p$, as long as $p \geq 3$ we have
$$
\binom{ap^3+p^3-1+b}{p^3-1} \equiv 0 \mod{p}.
$$
On the other hand, for~$a > 0$ and~$b = 0$, we have:
$$
\binom{ap^3+p^3-1+b}{p^3-1} \equiv 1 \mod{p}.
$$
We thus have~$\ext(\phi_\circ\rtimes(v,v')) \equiv 0 \mod{p}$ whenever~$|v| \neq |v'|$. When~$|v| = |v'|$, the binomial coefficient evaluates to $1$ modulo $p$, so that we have~$\ext(\phi_\circ\rtimes(v,v')) \equiv \ext\rtimes(\phi(v,v'))$, as desired. \qed
\subsection{ Proof of Lemma~\ref{TechnicalLemma}.}
\label{ProofTechnicalLemma}
Equations~(\textsf{1}) and~(\textsf{8}) follow immediately from the requirement that~$\phi$ be balanced.

We prove the lemma in the case where~$\phi$ is the \textsc{Swap}, \textsc{AndOr} or \textsc{TestEq} gate, and then explain how to adjust the proof when~$\phi$ is the \textsc{Identity}~gate. Write the input and output wires of~$\phi \wedge \phi'$ as~$v_1 \wedge v_1'$ and~$v_2 \wedge v_2'$, where here~$\wedge$ denotes concatenation,~$v_i \in \mathbb{F}_2^2$, and~$v_i' \in \mathbb{F}_2^k$, so that there are a total of~$k+2$ input and output wires in~$\phi \wedge \phi'$.

For every rearrangement~$\tau$ of~$(\phi \wedge \phi')\rtimes(v_1 \wedge v_1',v_2 \wedge v_2')$, we can restrict to a rearrangement of~$\phi_\circ'$ by looking only at the elements that come from~$\phi_\circ'$ in~$\tau$ (and shifting the elements back down to their previous values). Write this new rearrangement as~$\tau|_{\phi_\circ'}$. Then, writing~$\tau|_{\phi_\circ'}$ in one-line notation,~$\tau|_{\phi_\circ'}$ begins with some sequence (possibly of length zero) of input variables, and ends with some sequence (possibly of length zero) of output variables. Call the length of the first sequence~$\ell(\tau)$ and the length of the second sequence~$r(\tau)$.

Now we wish to begin with a permutation~$\tau' \leq \phi'\rtimes(v_1',v_2')$, and count the number of~$\tau \leq (\phi \wedge \phi')\rtimes(v_1 \wedge v_1',v_2 \wedge v_2')$ with~$\tau|_{\phi_\circ'} = \tau'$. Since we have fixed~$\tau'$, we need to consider only the possible ways that the elements of~$\phi$ can be rearranged with respect to each other or shuffled among the elements of~$\phi'$.

By Lemma~\ref{LimitingVariablesPassed}, we have:
$$\ext\bigl((\phi \wedge \phi')\rtimes(v_1 \wedge v_1',v_2 \wedge v_2')\bigr) \equiv 0 \mod{p}
$$ 
unless~$|v_1|+|v_1'| \equiv |v_2|+|v_2'| \mod{p^3}$.  Thus, we restrict our attention to the case where that holds. 
Then, by the same argument used in the proof of Lemma~\ref{LimitingVariablesPassed}, we have:
$$
|\phi'\rtimes(v_1',v_2')| = |\phi'| + (|v_1'|-k)-|v_2'| \equiv |v_1'|-|v_2'| \equiv |v_2|-|v_1| \mod{p^3}.
$$
Since~$\phi_\circ'$ adds a block of size~$p^3-1$, we have~$|\phi'_{\circ}\rtimes(v_1',v_2')| \equiv |v_2|-|v_1|-1 \mod{p^3}$.

We then note that the number of~$\tau \leq (\phi \wedge \phi')\rtimes(v_1 \wedge v_1',v_2 \wedge v_2')$ with~$\tau|_{\phi_\circ'} = \tau'$ is equal to:
$$
\frac{1}{\ell(\tau')!\.r(\tau')!} \cdot \ext\bigl(L^{\ell(\tau')}R^{r(\tau')}M^{|v_2|-|v_1|-\ell(\tau')-r(\tau')}(\phi)\rtimes (v_1,v_2)\bigr).
$$
With the~$\frac{1}{\ell(\tau')!\.r(\tau')!}$ term because repeated left and right insertion gives sets of consecutive decreasing elements that can be rearranged in~$\ell(\tau')!$ and~$r(\tau')!$ ways, respectively, but exactly one of these arrangements actually corresponds to~$\tau'$.

Next, we group permutations~$\tau'$ based on the value of~$\ell(\tau')$ and~$r(\tau')$. Let~$g(\ell,r)$ be the number of~$\tau' \leq \phi'\rtimes(v_1',v_2')$ with~$\ell(\tau') = \ell$ and~$r(\tau') = r$. Of course, the value~$g(\ell,r)$ also depends on~$\phi'\rtimes(v_1',v_2')$. We omit this dependence from our notation for the sake of readability.  We then have:
\begin{equation} \label{TechnicalLemmaEquation}
\ext\bigl((\phi \wedge \phi')\rtimes(v_1 \wedge v_1',v_2 \wedge v_2')\bigr) \, = \, \sum_{\ell,r \geq 0} \. \frac{g(\ell,r)}{\ell!\.r!} \cdot \ext\bigl(L^{\ell}R^{r}M^{|v_2|-|v_1|-\ell-r}(\phi)\rtimes(v_1,v_2)\bigr).
\end{equation}
It remains to compute~$g(\ell,r)$ for an arbitrary choice of~$v_1 \wedge v_1'$,~$v_2 \wedge v_2'$,~$\ell$, and~$r$. Beginning or ending~$\tau'$ with a particular sequence of elements is the same as counting the number of permutations of~$\tau'$ with those elements removed. Thus we have:
\begin{equation} \label{GEquation}
\ell!\, r!\sum_{|v_1'| - |w_1'|=\ell}\, \sum_{|w_2'|-|v_2'| = r}\ext\bigl(\phi'_{\circ}\rtimes(w_1',w_2')\bigr) \, = \sum_{\ell \leq h \leq k}\, \sum_{r \leq j \leq k} g(h,j).
\end{equation}
Here the left hand sum is taken over circuit states~$w_1'$ obtained from~$v_1'$ by flipping~$\ell$ wires from \textsc{true} to \textsc{false} and~$w_2'$ obtained from~$v_2'$ by flipping~$r$ from \textsc{false} to \textsc{true}. The~$\ell!$ and~$r!$ terms account for the ways to arrange the initial and final sequences, each strictly decreasing, of length~$\ell$ and~$r$ respectively.

Either of the wire flips just described reduces~$|v_1'|-|v_2'|$ by one, so that we have:
$$
|v_2|-|v_1| =|v_1'|-|v_2'| = \ell + r + |w_1'|-|w_2'|.
$$
Recall that, by Lemma~\ref{LimitingVariablesPassed}, for the left hand side of~\eqref{GEquation} to be nonzero modulo~$p$ we must have~$|w_1'|-|w_2'| = 0$.

Thus, on one hand, if~$\ell + r > |v_1'|+|v_2'|$, the left hand side of~\eqref{GEquation} is always~$0$ modulo~$p$, so that we conclude~$g(h,j) \equiv 0 \mod{p}$ for every~$h,j$ with~$h+j > \ell + r$.

On the other hand, if the left hand side of~\eqref{GEquation} is nonzero modulo~$p$, we have
$$
|v_2|-|v_1| =|v_1'|-|v_2'| = \ell + r \geq 0.
$$
Since~$v_1,v_2 \in \mathbb{F}_2^2$, we have~$|v_2| - |v_1| \leq 2$, and so we conclude~$0 \leq |v_1'|-|v_2'| \leq 2$.

Since~$\phi'$ is composed entirely of \textsc{Identity} and \textsc{Swap} gates, each input wire in~$v_1'$ can be matched with one output wire in~$v_2'$ whose signal state it controls. If we require~$\ext(\phi'_{\circ}\rtimes(w_1',w_2'))$ to be nonzero, then all the input-output wire pairs in~$w_1'$,~$w_2'$ match, and somewhere between zero and two input-output pairs of wires in~$v_1'$,~$v_2'$ have an input wire reading \textsc{true} and an output wire reading \textsc{false}. We refer to such a pair as a (\textsc{true}, \textsc{false}) pair and note that the number of (\textsc{true}, \textsc{false}) pairs is equal to 
$|v_1'|-|v_2'|$.

Of course, whenever~$|v_1'|-|v_2'| > 0$, the output wires do not correspond correctly to the input wires, so we need the count 
$$
\ext\bigl((\phi \wedge \phi')\rtimes(v_1 \wedge v_1',v_2 \wedge v_2')\bigr) \equiv 0 \mod{p}
$$ 
for this choice of 
$(v_1 \wedge v_1',v_2 \wedge v_2')$. We will now use~\eqref{TechnicalLemmaEquation} and~\eqref{GEquation} to do a computation which will produce the remaining equations given in the statement of this lemma.

Note that, because~$\phi'$ is made up of \textsc{Identity} and \textsc{Swap} gates, the technical restrictions in \S\ref{modpmod} are enough to ensure that~$\ext(\phi'_{\circ}\rtimes(w_1',w_2')) \equiv 1 \mod{p}$ whenever~$\ext(\phi'_{\circ}\rtimes(w_1',w_2'))$ is nonzero modulo~$p$.

\bigskip
\nin \textbf{\textit{Case 1}}: Zero (\textsc{true}, \textsc{false}) pairs in~$(v_1',v_2')$. Then~$\ext(\phi'_{\circ}\rtimes(w_1',w_2'))$ is nonzero modulo~$p$ precisely when~$w_1'=v_1'$ and~$w_2'=v_2'$, so that~\eqref{GEquation} gives~$g(0,0)\equiv 1 \mod{p}$ and~$g(h,j)\equiv 0 \mod{p}$ otherwise. Then~\eqref{TechnicalLemmaEquation} becomes:
$$
\ext\bigl((\phi \wedge \phi')\rtimes(v_1 \wedge v_1',v_2 \wedge v_2')\bigr) \equiv \ext\bigl(\phi\rtimes(v_1,v_2)\bigr) \mod{p}.
$$
So that~$\ext(\phi \wedge \phi')$ behaves like the Bruhat logic gate~$\phi$ on the top two wires.

\smallskip
\nin \textbf{\textit{Case 2}}: One (\textsc{true}, \textsc{false}) pair in~$(v_1',v_2')$. Then for~$\ell = 1$,~$r=0$, there is exactly one choice of~$w_1'$,~$w_2'$ with~$\ext(\phi'_{\circ}\rtimes(w_1',w_2'))$ nonzero modulo~$p$; this choice corresponds to switching the \textsc{true} input wire in the (\textsc{true}, \textsc{false}) pair to \textsc{false}. Likewise, for~$\ell=0$,~$r=1$, there is exactly one choice. Thus~\eqref{GEquation} gives~$g(1,0) \equiv g(0,1) \equiv 1 \mod{p}$.

Now taking~$\ell=r=0$, we have~$(w_1',w_2')=(v_1',v_2')$, and~$\ext(\phi'_{\circ}\rtimes(w_1',w_2')) \equiv 0 \mod{p}$ by Lemma~\ref{LimitingVariablesPassed}. Then~\eqref{GEquation} gives~$g(0,0) \equiv -2 \mod{p}$. There are four possible choices of~$v_1$ and~$v_2$ satisfying~$|v_2|-|v_1|=1$, which give~(\textsf{2}),~(\textsf{3}),~(\textsf{4}), and~(\textsf{5}).

\smallskip
\nin \textbf{\textit{Case 3}}: Two (\textsc{true}, \textsc{false}) pairs in~$(v_1',v_2')$. We proceed with a calculation similar to the one above. For~$\ell=2$,~$r=0$ and~$\ell = 0$,~$r=2$, there is exactly one choice of~$w_1'$,~$w_2'$, while for~$\ell=1$,~$r=1$, there are two choices. Then~\eqref{GEquation} gives:
$$
g(2,0)\equiv g(1,1) \equiv g(0,2)\equiv 2 \mod{p}.
$$
For~$\ell+r =1$ and~$\ell+r=0$, Lemma~\ref{LimitingVariablesPassed} tells us~$\ext(\phi'_{\circ}\rtimes(w_1',w_2')) \equiv 0 \mod{p}$, and we compute:
$$
g(1,0) \equiv g(0,1) \equiv -4 \mod{p} \quad \text{and} \quad g(0,0) \equiv 2 \mod{p}\ts.
$$
There is only one choice of~$v_1$ and~$v_2$ satisfying~$|v_2|-|v_1|=2$, which gives~(\textsf{6}).
\bigskip

This completes the proof when~$\phi$ is the \textsc{Swap}, \textsc{AndOr} or \textsc{TestEq} gate. When~$\phi$ is the \textsc{Identity} gate, we follow the same argument and find that we must only consider the cases when~$0 \leq |v_2|-|v_1| \leq 1$. Working through Case 1 and Case 2 above gives ~(\textsf{8}). \qed

\subsection{Proof of Lemma \ref{parametrization}}
\label{ProofParametrization}
We describe a function that sends a permutation~$\tau$ with ~$\tau \leq \phixz$ to a permutation~$\tau^* \leq \phi$. Note that since the blocks are in increasing order in~$\phixz$, they will still be in increasing order in~$\tau$. The only elements that can lie ``within'' one of these blocks (where ``within'' means to the right of some element from the block and to the left of another element of the block) are elements that were originally larger than the block and to its left, or smaller than the block and to its right.

To produce~$\tau^*$, push the elements of~$\tau$ that have moved within blocks out of their blocks, either to the left or right, back to the side they came from. We treat the penultimate block the same way. Then replace the blocks with the old~$x_i$'s and shift everything back down.

Since~$\phi$ has finite length, there are only finitely many choices of~$\tau^*$. For each~$\tau^*$ we count the number of possible~$\tau  \leq \phixz$ in the pre-image of the function described above. To do this computation, we consider the number of elements in~$\tau$ immediately to the left of an~$x_i$ and larger than it, or immediately to the right of an~$x_i$ and smaller than it. Call the first number~$\ell$ and the second~$r$. Then we are counting rearrangements of the block sequence
$$
\boxed{3}\ \boxed{2}\ \boxed{1}
$$
with blocks of lengths~$\ell$,~$z_i$ and~$r$, respectively, such that none of the elements from the~$\boxed{3}$ or~$\boxed{1}$ block cross the entire~$\boxed{2}$ block. We sum over the number~$h\leq \ell$ of elements that move from the~$\boxed{3}$ block into the~$\boxed{2}$ block, and the number~$j \leq r$ of elements that move from the~$\boxed{1}$ block into the~$\boxed{2}$ block:
$$
\sum_{0 \leq h \leq \ell} \sum_{0 \leq j \leq r} \binom{z_i+ h+j-2}{h} \binom{z_i+j-2}{j}.
$$
We thus obtain:
$$
\ext\left(\phixz\right) = \sum_{\tau^* \leq \phi}\,\prod_{i=1}^{t+1}\,\sum_{0 \leq h \leq \ell} \sum_{0 \leq j \leq r} \binom{z_i+ h+j-2}{h} \binom{z_i+j-2}{j},
$$
which is a polynomial in the~$z_i$'s and~$p^3-1$, as desired. \qed

\bigskip

\section{Height two posets} \label{HeightTwoPoset}

Let~$\cP=(X,<)$ be a poset on a set~$X$ of~$n$ elements~$\{x_1,\dots,x_n\}$.
Denote by~$\Ga=(X,E)$ its comparability graph,
with oriented edges~$(x_i,x_j)\in E$ if~$x_i<x_j$ in~$\cP$.
Denote by~$X'$ a identical copy of~$X$ with elements~$\{x_1',\dots,x_n'\}$.

Define the poset~$\cq=(X\cup X', \prec)$  on~$2n$ elements, by
having~$x_i\prec x_i'$ for all~$x_i \in X$, and~$x_i\prec x_j'$ for all
$x_i<x_j$, with~$x_i,x_j\in X$.  In particular, the Hasse diagram of~$\cq$
consists of~$n+|E|$ edges.  Note that~$\cq$ is a poset of height 2, see Figure~\ref{f:cq}.

\begin{figure}[hbt]
    \centering
    \begin{minipage}{0.45\textwidth}
        \centering
        \[\xymatrix@C=1pc@R=1pc{
& x_4\ar@{-}[dl] \ar@{-}[dr] & \\
x_2\ar@{-}[dr]  & & x_3\ar@{-}[dl] \\
& x_1 &}\]
        \caption{The Hasse diagram of a poset $\cP$.}
				\label{PosetFigure}
    \end{minipage}\hfill
    \begin{minipage}{0.45\textwidth}
        \centering
       \[\xymatrix@C=1pc@R=1pc{
x_1' & x_2' & x_3' & x_4' \\
x_1\ar@{-}[u] \ar@{-}[ur] \ar@{-}[urr] \ar@{-}[urrr] & x_2 \ar@{-}[u] \ar@{-}[urr] & x_3 \ar@{-}[u] \ar@{-}[ur] & x_4 \ar@{-}[u]}\]
        \caption{Poset $\cq$ associated to poset $\cP$.}
        \label{f:cq}
    \end{minipage}
\end{figure}

For every prime~$p$ between~$n$ and~$n^2$, we construct the modified poset~$\cq_p$ by adding, for all~$i$ and~$j$ satisfying $1 \leq i \leq n$ and $1 \leq j \leq p-2$, the element~$x_{ij}$ and the relation~$x_{ij} \prec x_i'$. Note that~$\cq_p$ is still of height~$2$ and has~$pn$ elements (see Figure~\ref{f:qp}).

We will use the number of linear extensions of~$\cq$ and~$\cq_p$ to compute the number of linear extensions of~$\cP$. Consider first the number~$\ext(\cq)$ of linear extensions of~$\cq$. Let $A \in \binom{[2n]}{n}$, i.e.\ $A$ is a $n$-subset of $[2n]=\{1,2,\dots,2n\}$.  Denote by~$\ext_A(\cq)$ be the number of linear extensions of~$\cq$ such that the values assigned to~$X'$ in the linear extension are the elements of~$A$. Then
$$
\ext(\cq) \, = \, \sum_{A\in \binom{[2n]}{n}} \. \ext_A(\cq)\ts.
$$

\begin{figure}[hbt]
\[\xymatrix@C=1pc@R=1pc{
& x_1' & & x_2' & & x_3' & & x_4' \\
x_{11}\ar@{-}[ur] & x_1\ar@{-}[u] \ar@{-}[urr] \ar@{-}[urrrr] \ar@{-}[urrrrrr] & x_{21}\ar@{-}[ur] & x_2 \ar@{-}[u] \ar@{-}[urrrr] &x_{31}\ar@{-}[ur] & x_3 \ar@{-}[u] \ar@{-}[urr] &x_{41}\ar@{-}[ur] & x_4 \ar@{-}[u]}\]
\caption{$\cq_p$ for $p=3$.}
\label{f:qp}
\end{figure}

\begin{lemma}
$\ext(\cP) = \ext_{\{2,4,6,\dots,2n\}}(\cq)$.
\end{lemma}
\begin{proof}
Since~$2,4,6,\dots$ are assigned to~$X'$, we must have~$1,3,5,\dots$ assigned to~$X$. An easy induction argument shows that if~$2k$ is assigned to~$x_i'$, then the element~$2k-1$ must be assigned to~$x_i$, for all $1\le k\le n$. The additional  relations on~$\cq$ ensure that this is a linear extension of~$\cq$ if and only if the corresponding assignment of values to~$X$ is a linear extension of~$\cP$.
\end{proof}

The above lemma should be compared with the following result:
\begin{lem}
$\ext(\cq_p) \equiv (-1)^n \ext_{\{2,4,6,\dots,2n\}}(\cq) \mod{p}$.
\end{lem}
\begin{proof}
Throughout the proof of this lemma, we will consider colorings of a set of integers. 
A \emph{coloring} is be a function from that set to some list of acceptable colors.

Let $A\in \binom{[2n]}{n}$, and write~$A = \{a_1,\dots,a_n\}$, with $a_1 < a_2 < \dots < a_n$. A coloring of the set $[pn]=\{1,2,\ldots,pn\}$ is called \emph{$A$-compatible} if the following conditions are satisfied:
\begin{enumerate}
\item
there is a sequence of~$2n$ integers~$b_1 < \dots < b_{2n}$ colored black,
\item
there are another~$n$ colors~$C_1,\dots,C_n$, and~$p-2$ integers are colored with each of these colors,
\item
all of the elements colored with~$C_k$ lie before~$b_{a_k}$.
\end{enumerate}

Let~$f_p(A)$ be the number of~$A$-compatible colorings of~$[pn]$. We observe that, given a linear extension of~$\cq$ where the values assigned to~$X'$ belong to the set~$A$, the number of linear extensions of~$\cq_p$ that preserve the ordering on~$X \cup X'$ is
$$
f_p(A) \bigl((p-2)!\bigr)^n.
$$
The~$b_k$'s represent the values assigned to~$X \cup X'$ in the linear extension of~$\cq_p$. Let~$x_i'$ be the element assigned the value $a_k$ in the given linear extension of~$\cq$. Then~$x_i'$ will be assigned~$b_{a_k}$ in the linear extension of~$\cq_p$, and the collection of elements colored with~$C_k$ represents the values assigned to the elements $x_{ij}$ attached to~$x_i'$. There are $(p-2)!$ ways to assign these values, for each $k$, with $1 \leq k \leq n$, giving the formula above.

We then have, by Wilson's theorem:
$$
\ext(\cq_p) \. = \. \left((p-2)!\right)^n \. \sum_{A\in \binom{[2n]}{n}} \ts \ext_A(\cq) \ts f_p(A) \. \equiv \. \sum_{A\in \binom{[2n]}{n}} \ts \ext_A(\cq) \ts f_p(A) \mod p.
$$
In an~$A$-compatible coloring of~$\{1,2,\dots,pn\}$, there are $a_k-1+k(p-2)$ terms to the left of~$b_{a_k}$ colored either black or one of the colors $C_1,\dots,C_k$. Among these terms, we can choose the position of the elements colored $C_k$ arbitrarily. This gives
$$
f_p(A) \. = \. \prod_{k=1}^{n} \ts \binom{a_k-1+k(p-2)}{p-2}\ts.
$$
For $A=\{2,4,6,\dots,2n\}$, we have $a_k=2k$, so this becomes
$$
f_p\left(\{2,4,6,\dots,2n\}\right) \, = \, \prod_{k=1}^n \. \binom{kp-1}{p-2} \, \equiv \. (-1)^n \mod{p}\ts.
$$
For every other~$A$ with~$\ext_A(\cq) \neq 0$, we have~$f_p(A) \equiv 0 \mod{p}$. 
Indeed, we must have $a_n=2n$, since~$2n \not \in X$. We proceed by induction on~$n-k$. 
Suppose that
$$
(a_{k+1},\dots,a_n) \. = \. (2k+2,\dots,2n).
$$
Then the relations~$x_i \prec x_i'$  force~$a_k$ to be equal to either~$2k$ or~$2k+1$. 
If~$a_k = 2k+1$, then
$$
\binom{a_k-1+kp-2k}{p-2}\. =\. \binom{kp}{p-2}
$$
will divide $f_p(A)$. Since $\binom{kp}{p-2} \equiv 0 \mod{p}$, we have $f_p(A) \equiv 0 \mod{p}$ unless $a_k=2k$.
\end{proof}

\begin{proof}[Proof of Theorem~\ref{HeightTwoTheorem}] Using the same Chinese Remainder Theorem argument we used in~\S\ref{PrimesSubsection}, the lemmas above show that computing $\ext(\cq_p)$ for the primes between $n$ and $n^2$ is sufficient to determine $\ext(\cp)$. Since \LE is \SP-complete, so is \htwoposet.
\end{proof}

\bigskip

\section{Incidence posets}
\label{s:inc-posets}

\subsection{Counting incidence posets}
Given a graph~$G=(V,E)$, we construct its incidence poset~$I_G$, with elements corresponding to vertices \emph{and} edges of~$G$, with~$x < y$ in~$\cp$ if and only if~$x \in E$,~$y \in V$ and~$y$ is an endpoint of~$x$. We write $\ext(G)$ for the number of linear extensions of $I_G$.

Our approach here is similar to our approach in Section~\ref{HeightTwoPoset}. We produce, given a poset~$\cp$ and a prime~$p > |\cp|$, a graph~$G_p(\cp)$ with:
$$
\ext\left(G_p(\cp)\right) \. \equiv \. (-1)^{|\cp|} \. \cdot \. 8\ts\ts\ext(\cp) \. \mod{p}\ts.
$$
Let~$G=(V,E)$ be a graph, with~$V = \{x_1,\dots,x_n\}$, and ~$\sigma \in S_n$ a permutation.  Denote by $\ext_{\sigma}(G)$ the number of linear extensions of~$I_G$, which satisfy the following condition: when restricted to~$V$, induce the permutation~$\sigma$, so that~$x_{\sigma^{-1}(1)} \leq x_{\sigma^{-1}(2)} \leq \cdots \leq x_{\sigma^{-1}(n)}$. We have:
$$
\ext(G) \, = \, \sum_{\sigma \in S_n} \.\ext_{\sigma}(G)\ts.
$$
Informally, to compute~$\ext_{\sigma}(G)$ we visit the vertices of~$G$ in the order dictated by~$\sigma$, accounting for the new edges we meet at each step.

Formally, given a permutation~$\sigma \in S_n$, we produce the sequence~$\{t_1,\dots,t_n\}$, where~$t_i$ is the number of edges in~$E$ with~$x_{\sigma^{-1}(i)}$ as an endpoint, and no endpoint~$x_{\sigma^{-1}(j)}$ for~$j < i$. Let $\{u_1,\dots,u_n\}$ be the sequence of partial sums of the $t_i$'s, so that
$$
u_k \. = \. t_1 \ts + \ts\ldots \ts +\ts t_k \ts.
$$
Note that~$u_k$ is the total number of edges incident to the set of vertices~$x_{\sigma^{-1}(1)},\dots,x_{\sigma^{-1}(k)}$.

Let~$|E| = m$. Then we call a coloring of the set~$\{1,2,\dots,m+n\}$~$(G,\sigma)$-compatible if the following conditions are satisfied:
\begin{enumerate}
\item
there is a sequence of~$n$ integers~$b_1 < \dots < b_{n}$ colored black,
\item
there are another~$n$ colors~$C_1,\dots,C_n$, and~$t_k$ integers are colored with the color~$C_k$\ts,
\item
all of the elements colored with~$C_k$ lie before~$b_{k}$\ts.
\end{enumerate}
Let~$f(G,\sigma)$ be the number of~$(G,\sigma)$-compatible colorings. In such a coloring, there are $u_k+k-1$ numbers to the left of $b_k$ colored either black or one of the colors $C_1,\dots,C_k$. Among these terms, we can choose the position of the elements colored $C_k$ arbitrarily. This gives:
$$
f(G,\sigma) \. = \. \prod_{k=1}^n \. \binom{u_k+k-1}{t_k}\ts.
$$
A~$(G,\sigma)$-compatible coloring corresponds to a collection of linear extensions of~$I_G$ counted by~$\ext_{\sigma}(G)$. The values assigned to the~$t_k$ new edges at~$x_{\sigma^{-1}(k)}$ are given by the numbers colored with~$C_k$, and these values can be assigned in~$(t_k)!$ ways, so that we have:
\begin{equation}
\label{IncidenceSum}
\ext(G) \, = \, \sum_{\sigma \in S_n}  f(G,\sigma) \. \prod_{k=1}^n \. (t_k)! 
\, = \, \sum_{\sigma \in S_n} \. \prod_{k=1}^n \. (t_k)! \.\binom{u_k+k-1}{t_k}\ts.
\end{equation}
In particular, when we are counting modulo $p$ we can restrict our attention to permutations~$\sigma$, which have corresponding sequences~$\{t_1,\dots,t_n\}$ with~$t_i < p$ for all~$i$. Informally, we want to visit each vertex of~$G$ in the order given by~$\sigma$, deleting the edges incident to each vertex after we visit it, and ensure that no vertex has at least~$p$ edges by the time we visit it.

Now we give the actual construction of~$G_p(\cp)$. The first step is to construct a gadget~$J_p$, which is a graph defined as follows. Start with the complete bipartite graph $K_{p-1,p-1}$ on~$2p-2$ vertices. Call these vertices~$y_1,\dots,y_{p-1}$ and~$z_1,\dots,z_{p-1}$ and add an additional~$p-2$ edges from~$z_{p-1}$ to~$z_i$ for $1 \leq i < p-1$. Note that each of the~$y_i$'s has degree~$p-1$ and the~$z_i$'s have degree~$\geq p$ (see Figure~\ref{f:jp}). We need the following:

\begin{lemma} \label{jp}
$\ext(J_p) \equiv -8 \mod{p}$.
\end{lemma}

We defer the proof of this lemma to the end of this section.

\begin{figure}[hbt]
    \centering
    \begin{minipage}{0.45\textwidth}
        \centering
        \[\xymatrix@C=1pc@R=1pc{
z_1\ar@{-}[r] & z_2 \\
y_1\ar@{-}[u] \ar@{-}[ur] & y_2\ar@{-}[u] \ar@{-}[ul] }\]
        \caption{$J_p$ for $p=3$.}\label{f:jp}
    \end{minipage}\hfill
    \begin{minipage}{0.45\textwidth}
        \centering
       \[\xymatrix@C=1pc@R=1pc{
z_1\ar@{-}[rr] && z_2 \\
y_1\ar@{-}[u] \ar@{-}[urr] && y_2\ar@{-}[u] \ar@{-}[ull] \\
& x_4\ar@{-}[dl] \ar@{-}[dr] \ar@{-}[ur] \ar@{-}[ul] & \\
x_2\ar@{-}[dr] \ar@{-}[uu] & & x_3\ar@{-}[dl] \ar@{-}[uu] \\
& x_1 &
}\]
        \caption{$G_p(\cP)$ for $\cP$ as in Figure \ref{PosetFigure} and $p=3$.}
        \label{f:Gp}
    \end{minipage}
\end{figure}

To construct~$G_p(\cp)$, add below~$J_p$ the Hasse diagram of~$\cp$ (treated as an undirected graph). For each element~$x \in \cp$, let~$v_x$ be the number of elements in~$\cp$ that cover~$x$. Add~$p-1-v_x$ edges from~$x$ to the degree~$p-1$ vertices~$y_i$ of~$J_p$ in an arbitrarily way (see Figure~\ref{f:Gp}).

\smallskip

Theorem~\ref{IncidenceTheorem} follows immediately from the following:

\begin{lemma}
$\ext(G_p(\cp)) \.\equiv\. (-1)^{|\cp|+1} \ts \cdot \ts 8\ts\ts\ext(\cp) \ts \mod{p}$
\end{lemma}
\begin{proof}
Every maximal element of~$\cp$ has~$v_x = 0$, and so is connected to each of the~$y_i$'s in~$J_p$. Since~$\cp$ has at least one maximal element, every element of~$J_p$ has degree~$\geq p$. Thus every~$\sigma$ which visits a vertex in~$J_p$ before visiting every maximal element of~$\cp$ has a term~$t_i \geq p$, so that~$\ext_{\sigma}(G_p(\cp)) \equiv 0 \mod{p}$. Likewise, of these permutations, every permutation~$\sigma$ that visits an element of~$\cp$ before visiting all of its immediate predecessors has~$\ext_{\sigma}(G_p(\cp)) \equiv 0 \mod{p}$.

Thus we can restrict our count of~$\ext(G_p(\cp))$ modulo $p$ to permutations that have as their first~$n$ terms a linear extension of~$\cp$. For these permutations, we have~$t_1=t_2=\ldots=t_n=p-1$, so that~$(t_k)! \equiv -1 \mod{p}$ by Wilson's theorem, and
$$
\binom{u_k+k-1}{t_k} \. = \. \binom{kp-1}{p-1} \. \equiv \. 1 \mod{p}.
$$
Furthermore, for every~$k > n$, we have~$t_1 + \ldots + t_k = np-n+(t_{n+1}+\cdots+t_k)+k-1$, so that
$$
\binom{u_k+k-1}{t_k} \. \equiv \. \binom{u_k-u_n + (k-n)-1}{t_k} \. \mod{p}.
$$
Now comparing the expressions for~$\ext(G_p(\cp))$ and~$\ext(J_p)$ given by~\eqref{IncidenceSum}, we have
$$
\ext\left(G_p(\cp)\right) \. \equiv \. (-1)^{|\cp|} \ext(\cp) \ext(J_p) \. \mod{p},
$$
and Lemma~\ref{jp} completes the proof.
\end{proof}
\begin{proof}[Proof of Theorem~\ref{IncidenceTheorem}]
Using the same Chinese Remainder Theorem argument we used in~\S\ref{PrimesSubsection} 
and Section~\ref{HeightTwoPoset}, the two lemmas above show that computing $\ext(G_p(P))$ for the primes between $|\cp|$ and~$|\cp|^2$ is sufficient to determine~$\ext(\cp)$. Since \LE is~$\SP$-Complete, so is~\incposet.
\end{proof}

\begin{figure}
\[
\xymatrix@C=1pc@R=1pc{
 & & &(4,3,1)\ar[dl]_0 \ar[dr]^1& & & & \\
 & &(4,2,1)\ar[dl]_0 \ar[dr]^1 & &(3,3,1)\ar[dl]_2 \ar[dr]^2 & & & \\
 &(4,1,1)\ar[dl]_0 \ar[dr]^3 & &(3,2,1)\ar[dl]_3 \ar[dr]^3 & &(2,3,1)\ar[dl]_3 \ar[dr]^3 & & \\
 (4,0,1) \ar[dr]^1 & &(3,1,1)\ar[dl]_4 \ar[dr]^1 & & (2,2,1)\ar[dl]_0 \ar[dr]^0 & &(1,3,1)\ar[dl]_1 \ar[dr]^4& \\
 &(3,0,1)\ar[dr]^3  & &(2,1,1)\ar[dl]_0 \ar[dr]^0 & &(1,2,1)\ar[dl]_0 \ar[dr]^0 & &(0,3,1) \ar[dl]_3 \\
 & &(2,0,1)\ar[dr]^1 & &(1,1,1)\ar[dl]_1 \ar[dr]^1 & &(0,2,1)\ar[dl]_1 & \\
 & & &(1,0,1) \ar[dr]^0 & &(0,1,1)\ar[dl]_0 & & \\
 & & & &(0,0,1) & & & }
\]
\caption{The $c=1$ half of the directed graph $\cg'$, with weights, for $p = 5$.}
\end{figure}
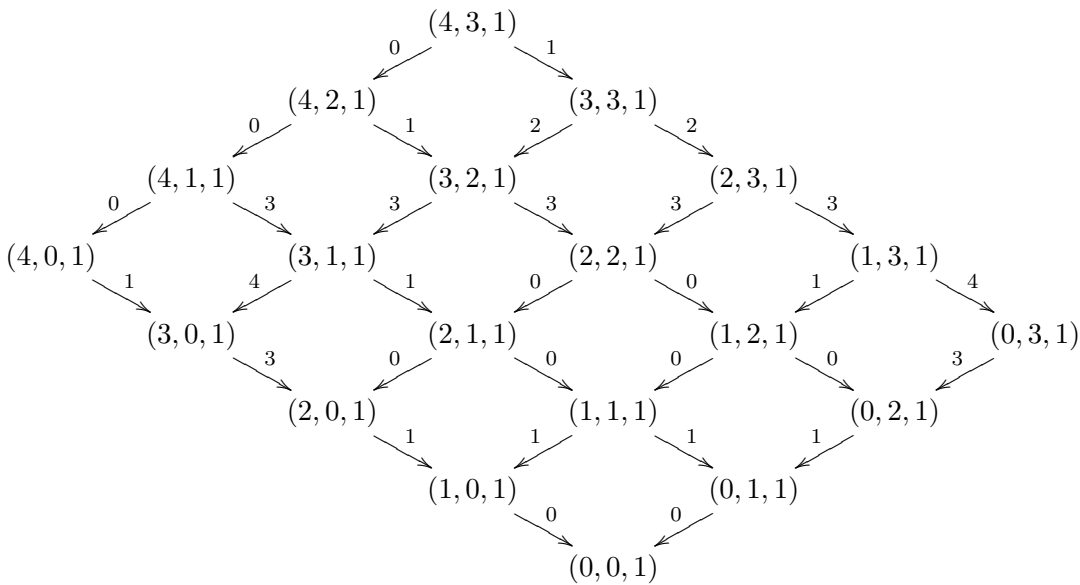

\subsection{Proof of Lemma~\ref{jp}}
Note that the values~$t_k$ and~$u_k+k-1$ in \eqref{IncidenceSum} are both independent of the order in which the previous~$k-1$ vertices are visited. They can be computed solely by identifying the vertex~$x_{\sigma^{-1}(k)}$ and the collection of vertices \ts $\{x_{\sigma^{-1}(i)}\}_{i < k}$. This motivates the following construction. Recall that the induced subgraphs of a graph~$G$ are those formed by deleting some vertices together with all incident edges. Take a directed graph~$\cg$ whose vertices are the induced subgraphs of~$J_p$ and whose edges point from each subgraph to those obtained from it by deleting a single vertex. Attach to each edge the weight
\begin{equation}\label{uequation}
(t_k)! \. \binom{u_k+k-1}{t_k} \. = \. (t_k)! \. \binom{u_k+k-1}{u_k-u_{k-1}}.
\end{equation}
Then~$\ext(J_p)$ is equal to the sum of all weighted paths in~$\cg$ from~$J_p$ to the empty subgraph. 

Let~$J_p(a,b,c)$ be an induced subgraph of~$J_p$ with~$a$ of the~$y_i$'s,~$b$ of the~$z_i$'s, for~$1 \leq i < p-1$, and~$c=1$ if~$z_{p-1} \in J_p(a,b,c)$,~$c=0$ otherwise, for~$0 \leq a \leq p-1$ and~$0 \leq b \leq p-2$. Since the~$y_i$'s, and the~$z_i$'s, except for~$z_{p-1}$, are indistinguishable, these subgraphs~$J_p(a,b,c)$ are \emph{all} of the induced subgraphs of~$J_p$, up to isomorphism.

We can thus reduce our graph of subgraphs~$\cg$ to the graph~$\cg'$ containing only these~$2p^2-2p$ vertices. We re-weight the edges from~$J_p(a,b,c)$ where~$a$, $b$ or~$c$ is reduced by one, by multiplying by~$a$, $b$ or~$c$, respectively. This accounts for the~$a$, $b$ or $c$ choices of vertex to remove. Write $\ell(a,b,c)$ for the value of $u_{k-1}+k-1$ upon reaching $J_p(a,b,c)$, that is, $\ell(a,b,c)$ is the number of vertices and edges that must be deleted from $J_p(p-1,p-2,1)$ to give $J_p(a,b,c)$. Then \eqref{uequation} gives the weight of the edge from $J_p(a,b,c)$ to $J_p(a-1,b,c)$ in terms of $a,b,c$ and $\ell$:
\begin{equation} \label{Jrecurrencea}
a \ts (b+c)! \. \binom{\ell(a,b,c)+b+c}{b+c} \, = \, a\ts (b+c)! \. \binom{\ell(a-1,b,c)-1}{\ell(a-1,b,c)-\ell(a,b,c)-1}\ts.
\end{equation}
The equations for the edges from $J_p(a,b,c)$ to $J_p(a,b-1,c)$ and $J_p(a,b,c-1)$ are the same up to a cyclic permutation of $(a,b,c)$.
The total number of edges in $J_p$ is $(p-1)^2+(p-2) = p^2-p-1$. The number of edges in $J_p(a,b,c)$ is $ab+ac+bc$, and we reach $J_p(a,b,c)$ by deleting $(p-1-a)+(p-2-b)+(1-c)$ vertices. We then calculate:
\begin{align*}
\ell(a,b,c) & \. = \. p^2-p-1 \ts - \ts (ab+(a+b)c) \ts + \ts (p-1-a)\ts +\ts (p-2-b)\ts +\ts (1-c) \\
&\. \equiv \. (a+2)(p-b-2)\ts + \ts (c-1)(a+b+2) \ts \mod{p}\ts.
\end{align*}

\begin{lemma}
When $c=1$, $(a+2)(p-b-2) > p$ and $(p-a-2)(b+2) > p$, every path in $\cg'$ that visits $J_p(a,b,c)$ has weight zero modulo $p$.
\end{lemma}
\begin{proof}
We argue by induction on $(2p-3)-(a+b)$, that is, on the distance in $\cg'$ from $J_p(p-1,p-2,1)$ to $J_p(a,b,c)$. When $a=p-1$, $b=p-2$, $c=1$, the conditions of the lemma are not met, and the statement is true vacuously.

Now suppose that $a,b,c$ satisfy the conditions in this lemma. Then a path that visits $J_p(a,b,c)$ must come from either $J_p(a+1,b,c)$ or $J_p(a,b+1,c)$. If the values $a+1,b,c$ satisfy the conditions in this lemma, we can then apply the induction hypothesis to show that every path through $J_p(a+1,b,c)$ has weight~$0$ modulo~$p$. In particular, a path that includes the edge from $J_p(a+1,b,c)$ to $J_p(a,b,c)$ has weight~$0$ modulo~$p$.

On the other hand, suppose that $a+1,b,c$ do not satisfy the conditions in this lemma. Then $(a+3)(p-b-2) > (a+2)(p-b-2) > p$, so we must have $(p-a-3)(b+2) \leq p$. Note that if $a$ or $b$ is greater than or equal to $p-2$, either $(a+2)(p-b-2) \leq 0$ or $(p-a-2)(b+2) \leq 0$. We thus have $a,b < p-2$, so that $(p-a-3)(b+2)=p$ is impossible.

However, when $(p-a-3)(b+2) < p$, since $b < p-2$, we have $(p-a-3)(b-2) >_p (p-a-2)(b+2)$. Thus, $\ell(a+1,b,c) >_p \ell(a,b,c)$, and so by~\eqref{Jrecurrencea}, the edge from $J_p(a+1,b,c)$ to $J_p(a,b,c)$ has weight~$0$ modulo~$p$. The argument for the edge from $J_p(a,b+1,c)$ to $J_p(a,b,c)$ is the same by symmetry.
\end{proof}

\begin{lemma}
Given $a,b$ with $(b+2)(p-a-2) \leq p$ the edge from $J_p(a,b,1)$ to $J_p(a,b,0)$ has weight~$0$ unless $a=p-3$ and $b=0$, $a=p-2$ and  $b=0$ or $1$, or $a=p-1$ with $b$ arbitrary. Similarly, given $a,b$ with $(a+2)(p-b-2) \leq p$, the edge from $J_p(a,b,1)$ to $J_p(a,b,0)$ has weight~$0$ unless $b=p-3$ and $a=0$, or $b=p-2$ and $a=0$ or $1$.
\end{lemma}
\begin{proof}
We give the proof of the first statement, since the proof of the second is essentially identical. Permuting~$(a,b,c)$ in~\eqref{Jrecurrencea} to find the weight of the edge from~$J_p(a,b,1)$ to~$J_p(a,b,0)$, we note that we must have $a+b < p$ and $a+b <_p a+b+\ell(a,b,1)$. Since $\ell(a,b,1) \equiv (b+2)(p-a-2) \mod{p}$, this gives:
\begin{align*}
a+b+(b+2)(p-a-2) \. &< \. p\ts.
\end{align*}
This implies that
\begin{align*}
p &< a+1+\tfrac{3}{b+1} \leq a+4.
\end{align*}
We conclude that $a > p-4$, and the rest of the lemma follows by elementary case analysis.
\end{proof}

\begin{proof}[Proof of Lemma~\ref{jp}] Note that the edges from $J_p(p-1,p-2,1)$ to $J_p(p-1,p-3,1)$ and $J_p(p-1,p-2,0)$ have weight~$0$ modulo~$p$. Combining this with the previous two lemmas, we conclude that every path in $\cg'$ has weight~$0$ modulo~$p$ unless it visits either $J_p(p-2,1,1)$ or $J_p(1,p-2,1)$. We now complete the desired calculation, through repeated applications of~\eqref{Jrecurrencea}, symmetry, and Wilson's theorem:
\begin{align*}
\ext(J_p(p-1,p-2,1)) \. &\. \equiv \. (p-1)(p-1)! \.\ext\bigl(J_p(p-2,p-2,1)\bigr) \\
&\. \equiv \. (p-2)! \. (-1)^{p-3} \ts \Bigl[\ext(J_p(p-2,1,1))\ts +\ts \ext\bigl(J_p(p-2,0,1)\bigr)\Bigr] \\
&\. \equiv  \. 2\. \ext\bigl(J_p(p-2,1,1)\bigr) \\
&\. \equiv  \. 2 \. (p-1)!\ts\Bigl[\ext\bigl(J_p(p-2,1,0)\bigr)\ts +\ts \ext\bigl(J_p(p-2,0,1)\bigr)\Bigr] \\
&\. \equiv  \. -4\. \ext\bigl(J_p(p-2,0,1)\bigr) \\
&\. \equiv  \. -4\. (p-2)!\. \ext\bigl(J_p(p-2,0,0)\bigr)\. - \. 4\. (p-2)\.\ts\ext\bigl(J_p(p-3,0,1)\bigr) \\
&\. \equiv  \. -4\. \ext\bigl(J_p(p-2,0,0)\bigr)\. + \. 8\ts\binom{p-1}{2}\ts\ext\bigl(J_p(p-3,0,0)\bigr) \\
&\. \equiv  \. -4\. (p-2)!\. + \. 4\. (p-1)(p-2)(p-3)! \\
&\. \equiv  \. -8 \. \mod{p}\ts.
\end{align*}
This completes the proof.
\end{proof}

\section{Final remarks} \label{s:fin-rem}

\subsection{} \label{ss:finrem-quotes}
Lee and Skipper~\cite{LS} report:

\smallskip\begin{center}\begin{minipage}{11cm}%
{[In personal communication]
\emph{``Brightwell and Winkler asserted that: \ts $(i)$~the complexity for the
general height-2 case is still open; \ts $(ii)$~there seems to be no work on
counting linear extensions of incidence posets; \ts $(iii)$~there is no
compelling reason to believe that the case of incidence posets should
be easier than general height-2 posets.''}}
\end{minipage}\end{center}

\smallskip

\nin
Now that we proved that both results are \SP-complete, this finally settles 
the debate.  Arguably, our proof of Theorem~\ref{HeightTwoTheorem} could
have been obtained 27 years ago when~\cite{BW1} appeared.  On the other 
hand, our proof of Theorem~\ref{Dim2Cor} was only made possible with 
advances in computer algebra and computer technology.  

\medskip

\subsection{} \label{ss:finrem-implem}
The equations in Appendix~\ref{GateEqs} are nonhomogeneous polynomials in~$5$ variables,
with a maximum total degree of~$5$. The coefficients are nonnegative integers $\leq 400$.
Before inserting parameters, the gates were permutations of length~$8$, so there were \ts 
$8!=40320$ \ts possibilities. In fact, the requirement that the variables be in strictly
decreasing order restricts the possibilities significantly. After some experimentation,
we added the further restriction that the first variable be in the first position of the permutation.
After these restrictions, only~$96$ possibilities remain.

For each gate, we generated the system of~12 polynomials in \ts {\tt C++}, for each of 
these~$96$ possible permutations.  We then computed which systems had solutions 
over~$\mathbb{C}$; the systems were tested with 
\ts {\tt Macaulay2}.\footnote{Computations were made with an 
{\tt Intel$^{\textregistered}$ Core\texttrademark\ i7-3610QM CPU
with 2.30GHz, 4~cores and 8Gb of RAM}.}
Generating the systems took~$314.4$ seconds, or an average of $3.3$ seconds per system. 
Testing all~$96$ systems took less than ten seconds for each of the
three gates. If we had needed to extend our search to~$6$ variables, the cost in
computing time would have increased significantly, as shown in Figure~\ref{computationtimefig}.

\smallskip

Here is the result of our computation.  For each gate, at least one of the of the~$96$ possible
permutations produced systems of equations with nontrivial solutions over $\mathbb{C}$.  To be precise:

\smallskip

$\dimo$ \ For the \textsc{Swap} gate, this worked for $2$ of the~$96$ possible permutations.


$\dimo$ \ For the \textsc{AndOr} gate, this worked for  $47$ of the~$96$ possible permutations.


$\dimo$ \ For the \textsc{TestEq} gate,  this worked for $4$ of the~$96$ possible permutations.

\begin{figure}[h]
  \centering
  \begin{tabular}{rrrl}
    Variables & Candidate gates & Computation time per candidate gate (sec.) \\\hline
    $4$ & $6$ & $\le 0.1$\\
    $5$ & $96$ & $3.3$ \\
    $6$ & $1200$ & 1618 ($\sim 27$ minutes)\\
  \end{tabular}
  \caption{Candidate permutations and computation time.}
  \label{computationtimefig}
\end{figure}

\medskip

\subsection{} \label{ss:finrem-complexity}
Let us quickly mention complexity implications of our results for people
unfamiliar with modern Complexity Theory.  Roughly, when a counting 
problem is \ts \SP-complete, this is an extremely strong evidence 
against it being computable in polynomial time, much stronger than $\P\ne\NP$, 
for example. Indeed, otherwise Toda's theorem \ts $\PH \subseteq \P^{\SP}$ \ts 
implies that every problem in \emph{polynomial hierarchy}~$\PH$ can be solved
in polynomial time.  

Another interesting question is about complexity of computing \dtwoposet~mod~$p$.  
Note that~\SP-completeness does not automatically imply the hardness of all such
problems, e.g. Per$(A)$~mod~$2$ of an integer matrix $A$ can be computed in 
polynomial time.  While our proof works implies hardness only for primes $\ge 11$, 
an early version of the proof works modulo~2, i.e.\ proves that \dtwoposet~mod~$2$ 
is $\oplus$\P-complete, see~\cite{Dit}.  In fact, we are confident that a larger version of our
construction would prove the result for the remaining primes $p\in\{3,5,7\}$.  

\medskip

\subsection{} \label{ss:finrem-mode}  Motivated by probabilistic applications, 
Mont\'ufar and Rauh~\cite{MoR} recently define the \emph{polytope of modes} 
$\rM(G,X)$, for every simple graph $G=(V,E)$ 
and independent subset of vertices $X\ssu V$. They prove that 
$$
\vol \. \rM(G,X) \, = \, \frac{\vol(\De^n)}{n!} \. \ext\bigl(P_{G,X}\bigr),
$$
where $n=|V|$, $\vol(\De^n)=\sqrt{n}/(n-1)!$, and $P_{G,X}$ is a height-2 poset 
with vertices in~$X$ on one level and $V\sm X$ on the other~\cite[Prop.~3]{MoR}
(see also~\cite{S1} for a strongly related \emph{order polytope}). 
The authors then discuss the problem of computing $\ext\bigl(P_{G,X}\bigr)$.

The following result follows easily from our Theorem~\ref{IncidenceTheorem}.
Curiously, we learned about this problem after the paper has been written.  

\begin{cor} \label{c:mode}
The problem of computing $\ext\bigl(P_{G,X}\bigr)$ is \SP-complete.  
\end{cor}

\begin{proof}
For a simple graph $H=(V,E)$, let $G$ be the \emph{medial graph} $G=M(H)$, 
defined as a graph on the set of vertices $V\cup E$ with edges given 
by adjacencies.  This is a bipartite graph, so $V$ is an independent set. 
By construction, $P_{G,X} = I_G$ is the incidence poset, which implies 
the result.  \end{proof}

\smallskip

\subsection{} \label{ss:finrem-open}
Let us mention some interesting open problems.  First, there is
a long tradition in Probabilistic Combinatorics to study properties of
\emph{random posets}, see survey~\cite{Bri}.  In fact, there are several
interesting models for random posets $P$ and for some of them rather
sharp results on the number $\ext(P)$ of their linear extensions
(see e.g.~\cite{BB} and references therein).  It would
be interesting to see if $\ext(P)$ can be computed in polynomial
time w.h.p.  We would be especially curious about complexity of computing
$\ext(P)$ for random height-2 posets, and of $\ext(P_\si)$ for random $\si \in S_n$.

We are also curious about variations on Theorem~\ref{BruhatTheorem}.
For example, is computing the size of the principal ideal of the \emph{strong Bruhat order}
\SP-complete?  What about other finite Coxeter groups?  We refer~\cite{Bre} for
definitions and the background.

Finally, we conjecture that computing the number $R(\si)$ of reduced factorizations 
of a permutation $\si \in S_n$ into adjacent transpositions is \SP-complete.  
Recall that $R(\si)$ can be computed in polynomial time in several special cases, 
see e.g.~\cite{MPP2}.  Note that such factorizations can be viewed as saturated 
chains $1 \to \si$ in the weak Bruhat order $B_n=(S_n,\le)$.

\smallskip

\subsection{} \label{ss:finrem-rival}
In his 1984 short survey paper~\cite{Riv}, Rival wrote:

\smallskip\begin{center}\begin{minipage}{11cm}%
{\emph{``Counting is hard! And counting the linear extensions of an ordered set
is no exception. This counting problem is  tractable only for some special
and very simple classes of ordered sets.''}}
\end{minipage}\end{center}

\smallskip

\nin
While we now know many more classes of posets for which one can compute
the number of linear extensions, the sentiment continues to hold.

\vskip.55cm

\subsection*{Acknowledgements}
We are grateful to Greg Kuperberg, Laci M.\ Lov\'asz, Alejandro Morales, 
Greta Panova, Bruce Rothschild, Pete Winkler and Damir Yeliussizov for 
helpful conversations and remarks on the subject.  We are thankful to 
Jon Lee for telling us about incidence posets and bringing~\cite{LS} 
to our attention.  Guido Mont\'ufar kindly showed us~\cite{MoR} and explained 
the problem discussed in~$\S$\ref{ss:finrem-mode}. We are especially 
grateful to Anton Leykin for his insights into computer algebra and 
his help programming in {\tt Macaulay2}, and to MSRI for hosting
us to make such conversations possible.

Finally, the second author owes a debt of gratitude to Ivan Rival,
who went out of his way to meet us during his travels to the Soviet
Union, when we were still an undergraduate interested in combinatorics.
Ivan encouraged us to work on posets, an advice we didn't adhere until now.

The second author was partially supported by MSRI and the~NSF.


\newpage

\appendix

\section{Gate equations} \label{A:gate-eq}
\label{GateEqs}
We print here the systems of polynomial equations for the parametrized \textsc{Swap},
\textsc{AndOr}, and~\textsc{TestEq} gates given in Lemma~\ref{Gates}.
For each of these gates, there are six equations from Lemma~\ref{TechnicalLemma}
and six equations from the requirements for the logical operation of the gate itself, for a total of twelve equations.

\medskip

\subsection{\normalsize \textsc{Swap} gate.}
\label{swapeqn}

\begin{footnotesize}


\begin{enumerate}

\item $|\phi|-2 \equiv 0 \mod{p^3}$:

$
z_1+z_2+z_3+z_4+ z_5+4 = 0
$

\smallskip

\item $\ext\bigl(\phi\rtimes(11,11)\bigr) \equiv 1 \mod p$:

$
2 z_2 z_5^3+2 z_1 z_5^3+4 z_5^3+6 z_2 z_4 z_5^2+6 z_1 z_4 z_5^2+12 z_4 z_5^2+3 z_2 z_3 z_5^2+3 z_1 z_3 z_5^2+6 z_3 z_5^2+3 z_2^2 z_5^2+3 z_1 z_2 z_5^2+15 z_2 z_5^2+6 z_1 z_5^2+15 z_5^2+6 z_2 z_4^2 z_5+6 z_1 z_4^2 z_5+12 z_4^2 z_5+6 z_2 z_3 z_4 z_5+6 z_1 z_3 z_4 z_5+12 z_3 z_4 z_5+6 z_2^2 z_4 z_5+6 z_1 z_2 z_4 z_5+30 z_2 z_4 z_5+12 z_1 z_4 z_5+30 z_4 z_5+3 z_2 z_3 z_5+3 z_1 z_3 z_5+6 z_3 z_5+3 z_2^2 z_5+3 z_1 z_2 z_5+13 z_2 z_5+4 z_1 z_5+11 z_5 = 6
$

\smallskip

\item $\ext\bigl(\phi\rtimes(10,01)\bigr) \equiv 1 \mod p$:

$2 z_5^3+6 z_4 z_5^2+3 z_3 z_5^2+3 z_2 z_5^2+12 z_5^2+6 z_4^2 z_5+6 z_3 z_4 z_5+6 z_2 z_4 z_5+24 z_4 z_5+9 z_3 z_5+9 z_2 z_5+16 z_5+6 z_4^2+6 z_3 z_4+6 z_2 z_4+12 z_4 = 6$

\smallskip

\item $\ext\bigl(\phi\rtimes(10,10)\bigr) \equiv 0 \mod p$:

$2 z_5^3+6 z_4 z_5^2+3 z_3 z_5^2+3 z_2 z_5^2+12 z_5^2+6 z_4^2 z_5+6 z_3 z_4 z_5+6 z_2 z_4 z_5+24 z_4 z_5+9 z_3 z_5+9 z_2 z_5+22 z_5+6 z_4^2+6 z_3 z_4+6 z_2 z_4+18 z_4+6 z_3+6 z_2+12 = 0$

\smallskip

\item $\ext\bigl(\phi\rtimes(01,10)\bigr) \equiv 1 \mod p$:

$2 z_2 z_5^3+2 z_1 z_5^3+2 z_5^3+6 z_2 z_4 z_5^2+6 z_1 z_4 z_5^2+6 z_4 z_5^2+3 z_2 z_3 z_5^2+3 z_1 z_3 z_5^2+3 z_3 z_5^2+3 z_2^2 z_5^2+3 z_1 z_2 z_5^2+18 z_2 z_5^2+12 z_1 z_5^2+15 z_5^2+6 z_2 z_4^2 z_5+6 z_1 z_4^2 z_5+6 z_4^2 z_5+6 z_2 z_3 z_4 z_5+6 z_1 z_3 z_4 z_5+6 z_3 z_4 z_5+6 z_2^2 z_4 z_5+6 z_1 z_2 z_4 z_5+36 z_2 z_4 z_5+24 z_1 z_4 z_5+30 z_4 z_5+9 z_2 z_3 z_5+9 z_1 z_3 z_5+9 z_3 z_5+9 z_2^2 z_5+9 z_1 z_2 z_5+40 z_2 z_5+22 z_1 z_5+31 z_5+6 z_2 z_4^2+6 z_1 z_4^2+6 z_4^2+6 z_2 z_3 z_4+6 z_1 z_3 z_4+6 z_3 z_4+6 z_2^2 z_4+6 z_1 z_2 z_4+30 z_2 z_4+18 z_1 z_4+24 z_4+6 z_2 z_3+6 z_1 z_3+6 z_3+6 z_2^2+6 z_1 z_2+24 z_2+12 z_1+18 = 6$

\smallskip

\item $\ext(\phi\rtimes(01,01)\bigr) \equiv 0 \mod p$:

$2 z_2 z_5^3+2 z_1 z_5^3+2 z_5^3+6 z_2 z_4 z_5^2+6 z_1 z_4 z_5^2+6 z_4 z_5^2+3 z_2 z_3 z_5^2+3 z_1 z_3 z_5^2+3 z_3 z_5^2+3 z_2^2 z_5^2+3 z_1 z_2 z_5^2+18 z_2 z_5^2+12 z_1 z_5^2+15 z_5^2+6 z_2 z_4^2 z_5+6 z_1 z_4^2 z_5+6 z_4^2 z_5+6 z_2 z_3 z_4 z_5+6 z_1 z_3 z_4 z_5+6 z_3 z_4 z_5+6 z_2^2 z_4 z_5+6 z_1 z_2 z_4 z_5+36 z_2 z_4 z_5+24 z_1 z_4 z_5+30 z_4 z_5+9 z_2 z_3 z_5+9 z_1 z_3 z_5+9 z_3 z_5+9 z_2^2 z_5+9 z_1 z_2 z_5+34 z_2 z_5+16 z_1 z_5+25 z_5+6 z_2 z_4^2+6 z_1 z_4^2+6 z_4^2+6 z_2 z_3 z_4+6 z_1 z_3 z_4+6 z_3 z_4+6 z_2^2 z_4+6 z_1 z_2 z_4+24 z_2 z_4+12 z_1 z_4+18 z_4 = 0$

\smallskip

\item $\ext(\phi\rtimes(00,00)\bigr) \equiv 1 \mod p$:

$2 z_5^3+6 z_4 z_5^2+3 z_3 z_5^2+3 z_2 z_5^2+18 z_5^2+6 z_4^2 z_5+6 z_3 z_4 z_5+6 z_2 z_4 z_5+36 z_4 z_5+15 z_3 z_5+15 z_2 z_5+40 z_5+12 z_4^2+12 z_3 z_4+12 z_2 z_4+36 z_4+6 z_3+6 z_2+15 = 3$

\smallskip

\item $-2 \ext\bigl(M(\phi_\circ)\rtimes(10,11)\bigr) + \ext\bigl(L(\phi_\circ)\rtimes(10,11)\bigr) + \ext\bigl(R(\phi_\circ)\rtimes(10,11)\bigr) \equiv 0 \mod{p}$:

$2 z_5^4+8 z_4 z_5^3+5 z_3 z_5^3+5 z_2 z_5^3+2 z_1 z_5^3+12 z_5^3+12 z_4^2 z_5^2+15 z_3 z_4 z_5^2+15 z_2 z_4 z_5^2+6 z_1 z_4 z_5^2+36 z_4 z_5^2+3 z_3^2 z_5^2+6 z_2 z_3 z_5^2+3 z_1 z_3 z_5^2+18 z_3 z_5^2+3 z_2^2 z_5^2+3 z_1 z_2 z_5^2+18 z_2 z_5^2+6 z_1 z_5^2+22 z_5^2+6 z_4^3 z_5+12 z_3 z_4^2 z_5+12 z_2 z_4^2 z_5+6 z_1 z_4^2 z_5+30 z_4^2 z_5+6 z_3^2 z_4 z_5+12 z_2 z_3 z_4 z_5+6 z_1 z_3 z_4 z_5+33 z_3 z_4 z_5+6 z_2^2 z_4 z_5+6 z_1 z_2 z_4 z_5+33 z_2 z_4 z_5+12 z_1 z_4 z_5+40 z_4 z_5+3 z_3^2 z_5+6 z_2 z_3 z_5+3 z_1 z_3 z_5+13 z_3 z_5+3 z_2^2 z_5+3 z_1 z_2 z_5+13 z_2 z_5+4 z_1 z_5+12 z_5 = 0$

\smallskip

\item $-2 \ext\bigl(M(\phi_\circ)\rtimes(01,11)\bigr) + \ext\bigl(L(\phi_\circ)\rtimes(01,11)\bigr) + \ext\bigl(R(\phi_\circ)\rtimes(01,11)\bigr) \equiv 0 \mod{p}$:

$2 z_2 z_5^4+2 z_1 z_5^4+2 z_5^4+8 z_2 z_4 z_5^3+8 z_1 z_4 z_5^3+8 z_4 z_5^3+5 z_2 z_3 z_5^3+5 z_1 z_3 z_5^3+5 z_3 z_5^3+5 z_2^2 z_5^3+7 z_1 z_2 z_5^3+22 z_2 z_5^3+2 z_1^2 z_5^3+16 z_1 z_5^3+17 z_5^3+12 z_2 z_4^2 z_5^2+12 z_1 z_4^2 z_5^2+12 z_4^2 z_5^2+15 z_2 z_3 z_4 z_5^2+15 z_1 z_3 z_4 z_5^2+15 z_3 z_4 z_5^2+15 z_2^2 z_4 z_5^2+21 z_1 z_2 z_4 z_5^2+66 z_2 z_4 z_5^2+6 z_1^2 z_4 z_5^2+48 z_1 z_4 z_5^2+51 z_4 z_5^2+3 z_2 z_3^2 z_5^2+3 z_1 z_3^2 z_5^2+3 z_3^2 z_5^2+6 z_2^2 z_3 z_5^2+9 z_1 z_2 z_3 z_5^2+30 z_2 z_3 z_5^2+3 z_1^2 z_3 z_5^2+24 z_1 z_3 z_5^2+24 z_3 z_5^2+3 z_2^3 z_5^2+6 z_1 z_2^2 z_5^2+27 z_2^2 z_5^2+3 z_1^2 z_2 z_5^2+33 z_1 z_2 z_5^2+67 z_2 z_5^2+6 z_1^2 z_5^2+37 z_1 z_5^2+43 z_5^2+6 z_2 z_4^3 z_5+6 z_1 z_4^3 z_5+6 z_4^3 z_5+12 z_2 z_3 z_4^2 z_5+12 z_1 z_3 z_4^2 z_5+12 z_3 z_4^2 z_5+12 z_2^2 z_4^2 z_5+18 z_1 z_2 z_4^2 z_5+54 z_2 z_4^2 z_5+6 z_1^2 z_4^2 z_5+42 z_1 z_4^2 z_5+42 z_4^2 z_5+6 z_2 z_3^2 z_4 z_5+6 z_1 z_3^2 z_4 z_5+6 z_3^2 z_4 z_5+12 z_2^2 z_3 z_4 z_5+18 z_1 z_2 z_3 z_4 z_5+57 z_2 z_3 z_4 z_5+6 z_1^2 z_3 z_4 z_5+45 z_1 z_3 z_4 z_5+45 z_3 z_4 z_5+6 z_2^3 z_4 z_5+12 z_1 z_2^2 z_4 z_5+51 z_2^2 z_4 z_5+6 z_1^2 z_2 z_4 z_5+63 z_1 z_2 z_4 z_5+124 z_2 z_4 z_5+12 z_1^2 z_4 z_5+70 z_1 z_4 z_5+79 z_4 z_5+3 z_2 z_3^2 z_5+3 z_1 z_3^2 z_5+3 z_3^2 z_5+6 z_2^2 z_3 z_5+9 z_1 z_2 z_3 z_5+25 z_2 z_3 z_5+3 z_1^2 z_3 z_5+19 z_1 z_3 z_5+19 z_3 z_5+3 z_2^3 z_5+6 z_1 z_2^2 z_5+22 z_2^2 z_5+3 z_1^2 z_2 z_5+26 z_1 z_2 z_5+47 z_2 z_5+4 z_1^2 z_5+23 z_1 z_5+28 z_5 = 0$

\smallskip

\item$-2 \ext\bigl(M(\phi_\circ)\rtimes(00,01)\bigr) + \ext\bigl(L(\phi_\circ)\rtimes(00,01)\bigr) + \ext\bigl(R(\phi_\circ)\rtimes(00,01)\bigr) \equiv 0 \mod{p}$:

$2 z_5^4+8 z_4 z_5^3+5 z_3 z_5^3+5 z_2 z_5^3+2 z_1 z_5^3+20 z_5^3+12 z_4^2 z_5^2+15 z_3 z_4 z_5^2+15 z_2 z_4 z_5^2+6 z_1 z_4 z_5^2+60 z_4 z_5^2+3 z_3^2 z_5^2+6 z_2 z_3 z_5^2+3 z_1 z_3 z_5^2+33 z_3 z_5^2+3 z_2^2 z_5^2+3 z_1 z_2 z_5^2+33 z_2 z_5^2+12 z_1 z_5^2+64 z_5^2+6 z_4^3 z_5+12 z_3 z_4^2 z_5+12 z_2 z_4^2 z_5+6 z_1 z_4^2 z_5+54 z_4^2 z_5+6 z_3^2 z_4 z_5+12 z_2 z_3 z_4 z_5+6 z_1 z_3 z_4 z_5+63 z_3 z_4 z_5+6 z_2^2 z_4 z_5+6 z_1 z_2 z_4 z_5+63 z_2 z_4 z_5+24 z_1 z_4 z_5+124 z_4 z_5+9 z_3^2 z_5+18 z_2 z_3 z_5+9 z_1 z_3 z_5+52 z_3 z_5+9 z_2^2 z_5+9 z_1 z_2 z_5+52 z_2 z_5+16 z_1 z_5+64 z_5+6 z_4^3+12 z_3 z_4^2+12 z_2 z_4^2+6 z_1 z_4^2+36 z_4^2+6 z_3^2 z_4+12 z_2 z_3 z_4+6 z_1 z_3 z_4+36 z_3 z_4+6 z_2^2 z_4+6 z_1 z_2 z_4+36 z_2 z_4+12 z_1 z_4+48 z_4 = 0$

\smallskip

\item $-2 \ext\bigl(M(\phi_\circ)\rtimes(00,10)\bigr) + \ext\bigl(L(\phi_\circ)\rtimes(00,10)\bigr) + \ext\bigl(R(\phi_\circ)\rtimes(00,10)\bigr) \equiv 0 \mod{p}$:

$2 z_5^4+8 z_4 z_5^3+5 z_3 z_5^3+5 z_2 z_5^3+2 z_1 z_5^3+20 z_5^3+12 z_4^2 z_5^2+15 z_3 z_4 z_5^2+15 z_2 z_4 z_5^2+6 z_1 z_4 z_5^2+60 z_4 z_5^2+3 z_3^2 z_5^2+6 z_2 z_3 z_5^2+3 z_1 z_3 z_5^2+33 z_3 z_5^2+3 z_2^2 z_5^2+3 z_1 z_2 z_5^2+33 z_2 z_5^2+12 z_1 z_5^2+70 z_5^2+6 z_4^3 z_5+12 z_3 z_4^2 z_5+12 z_2 z_4^2 z_5+6 z_1 z_4^2 z_5+54 z_4^2 z_5+6 z_3^2 z_4 z_5+12 z_2 z_3 z_4 z_5+6 z_1 z_3 z_4 z_5+63 z_3 z_4 z_5+6 z_2^2 z_4 z_5+6 z_1 z_2 z_4 z_5+63 z_2 z_4 z_5+24 z_1 z_4 z_5+136 z_4 z_5+9 z_3^2 z_5+18 z_2 z_3 z_5+9 z_1 z_3 z_5+64 z_3 z_5+9 z_2^2 z_5+9 z_1 z_2 z_5+64 z_2 z_5+22 z_1 z_5+100 z_5+6 z_4^3+12 z_3 z_4^2+12 z_2 z_4^2+6 z_1 z_4^2+42 z_4^2+6 z_3^2 z_4+12 z_2 z_3 z_4+6 z_1 z_3 z_4+48 z_3 z_4+6 z_2^2 z_4+6 z_1 z_2 z_4+48 z_2 z_4+18 z_1 z_4+84 z_4+6 z_3^2+12 z_2 z_3+6 z_1 z_3+36 z_3+6 z_2^2+6 z_1 z_2+36 z_2+12 z_1+48 = 0$

\smallskip

\item $2\ext\bigl(M^2(\phi_\circ)\rtimes(00,11)\bigr) - 4 \ext\bigl(LM(\phi_\circ)\rtimes(00,11)\bigr) - 4\ext\bigl(RM(\phi_\circ)\rtimes(00,11)\bigr) +$

$\ext\bigl(L^2(\phi_\circ)\rtimes(00,11)\bigr) + 2\ext\bigl(LR(\phi_\circ)\rtimes(00,11)\bigr) + \ext\bigl(R^2(\phi_\circ)\rtimes(00,11)\bigr) \equiv 0 \mod{p}$:

$2 z_5^5+10 z_4 z_5^4+7 z_3 z_5^4+7 z_2 z_5^4+4 z_1 z_5^4+20 z_5^4+20 z_4^2 z_5^3+28 z_3 z_4 z_5^3+28 z_2 z_4 z_5^3+16 z_1 z_4 z_5^3+80 z_4 z_5^3+8 z_3^2 z_5^3+16 z_2 z_3 z_5^3+10 z_1 z_3 z_5^3+50 z_3 z_5^3+8 z_2^2 z_5^3+10 z_1 z_2 z_5^3+50 z_2 z_5^3+2 z_1^2 z_5^3+26 z_1 z_5^3+70 z_5^3+18 z_4^3 z_5^2+39 z_3 z_4^2 z_5^2+39 z_2 z_4^2 z_5^2+24 z_1 z_4^2 z_5^2+114 z_4^2 z_5^2+24 z_3^2 z_4 z_5^2+48 z_2 z_3 z_4 z_5^2+30 z_1 z_3 z_4 z_5^2+147 z_3 z_4 z_5^2+24 z_2^2 z_4 z_5^2+30 z_1 z_2 z_4 z_5^2+147 z_2 z_4 z_5^2+6 z_1^2 z_4 z_5^2+78 z_1 z_4 z_5^2+206 z_4 z_5^2+3 z_3^3 z_5^2+9 z_2 z_3^2 z_5^2+6 z_1 z_3^2 z_5^2+33 z_3^2 z_5^2+9 z_2^2 z_3 z_5^2+12 z_1 z_2 z_3 z_5^2+66 z_2 z_3 z_5^2+3 z_1^2 z_3 z_5^2+39 z_1 z_3 z_5^2+107 z_3 z_5^2+3 z_2^3 z_5^2+6 z_1 z_2^2 z_5^2+33 z_2^2 z_5^2+3 z_1^2 z_2 z_5^2+39 z_1 z_2 z_5^2+107 z_2 z_5^2+6 z_1^2 z_5^2+50 z_1 z_5^2+100 z_5^2+6 z_4^4 z_5+18 z_3 z_4^3 z_5+18 z_2 z_4^3 z_5+12 z_1 z_4^3 z_5+54 z_4^3 z_5+18 z_3^2 z_4^2 z_5+36 z_2 z_3 z_4^2 z_5+24 z_1 z_3 z_4^2 z_5+111 z_3 z_4^2 z_5+18 z_2^2 z_4^2 z_5+24 z_1 z_2 z_4^2 z_5+111 z_2 z_4^2 z_5+6 z_1^2 z_4^2 z_5+66 z_1 z_4^2 z_5+160 z_4^2 z_5+6 z_3^3 z_4 z_5+18 z_2 z_3^2 z_4 z_5+12 z_1 z_3^2 z_4 z_5+60 z_3^2 z_4 z_5+18 z_2^2 z_3 z_4 z_5+24 z_1 z_2 z_3 z_4 z_5+120 z_2 z_3 z_4 z_5+6 z_1^2 z_3 z_4 z_5+72 z_1 z_3 z_4 z_5+185 z_3 z_4 z_5+6 z_2^3 z_4 z_5+12 z_1 z_2^2 z_4 z_5+60 z_2^2 z_4 z_5+6 z_1^2 z_2 z_4 z_5+72 z_1 z_2 z_4 z_5+185 z_2 z_4 z_5+12 z_1^2 z_4 z_5+92 z_1 z_4 z_5+172 z_4 z_5+3 z_3^3 z_5+9 z_2 z_3^2 z_5+6 z_1 z_3^2 z_5+25 z_3^2 z_5+9 z_2^2 z_3 z_5+12 z_1 z_2 z_3 z_5+50 z_2 z_3 z_5+3 z_1^2 z_3 z_5+29 z_1 z_3 z_5+64 z_3 z_5+3 z_2^3 z_5+6 z_1 z_2^2 z_5+25 z_2^2 z_5+3 z_1^2 z_2 z_5+29 z_1 z_2 z_5+64 z_2 z_5+4 z_1^2 z_5+28 z_1 z_5+48 z_5 = 0$
\end{enumerate}

\medskip
\textsc{SOLUTION}:$\. (z_1,z_2,z_3,z_4,z_5) = (-1,-2,0,1,-2).$
\end{footnotesize}

\bigskip

\subsection{\normalsize \textsc{AndOr} gate.}
\label{andoreqn}

\begin{footnotesize}

\smallskip

\begin{enumerate}

\item $|\phi|-2 \equiv 0 \mod{p^3}$:

$
z_1+z_2+z_3+z_4+z_5 +4=0
$

\smallskip

\item $\ext\bigl(\phi\rtimes(11,11)\bigr) \not \equiv 0 \mod p$:

$z_3 z_5^3+z_2 z_5^3+z_1 z_5^3+2 z_5^3+2 z_3 z_4 z_5^2+2 z_2 z_4 z_5^2+2 z_1 z_4 z_5^2+4 z_4 z_5^2+2 z_3^2 z_5^2+3 z_2 z_3 z_5^2+2 z_1 z_3 z_5^2+9 z_3 z_5^2+z_2^2 z_5^2+z_1 z_2 z_5^2+6 z_2 z_5^2+3 z_1 z_5^2+8 z_5^2+z_3 z_4^2 z_5+z_2 z_4^2 z_5+z_1 z_4^2 z_5+2 z_4^2 z_5+2 z_3^2 z_4 z_5+3 z_2 z_3 z_4 z_5+2 z_1 z_3 z_4 z_5+9 z_3 z_4 z_5+z_2^2 z_4 z_5+z_1 z_2 z_4 z_5+6 z_2 z_4 z_5+3 z_1 z_4 z_5+8 z_4 z_5+z_3^3 z_5+2 z_2 z_3^2 z_5+z_1 z_3^2 z_5+7 z_3^2 z_5+z_2^2 z_3 z_5+z_1 z_2 z_3 z_5+8 z_2 z_3 z_5+3 z_1 z_3 z_5+14 z_3 z_5+z_2^2 z_5+z_1 z_2 z_5+6 z_2 z_5+2 z_1 z_5+8 z_5 \neq 0$

\smallskip

\item $\ext\bigl(\phi\rtimes(10,01)\bigr) \not \equiv 0 \mod p$:

$z_5^3+2 z_4 z_5^2+2 z_3 z_5^2+z_2 z_5^2+4 z_5^2+z_4^2 z_5+2 z_3 z_4 z_5+z_2 z_4 z_5+4 z_4 z_5+z_3^2 z_5+z_2 z_3 z_5+4 z_3 z_5+z_2 z_5+3 z_5 \neq 0$

\smallskip

\item $\ext\bigl(\phi\rtimes(10,10)\bigr)  \equiv 0 \mod p$:

$z_5^3+2 z_4 z_5^2+2 z_3 z_5^2+z_2 z_5^2+6 z_5^2+z_4^2 z_5+2 z_3 z_4 z_5+z_2 z_4 z_5+7 z_4 z_5+z_3^2 z_5+z_2 z_3 z_5+7 z_3 z_5+3 z_2 z_5+11 z_5+z_4^2+2 z_3 z_4+z_2 z_4+5 z_4+z_3^2+z_2 z_3+5 z_3+2 z_2+6 = 0$

\smallskip

\item $\ext\bigl(\phi\rtimes(01,10)\bigr) \equiv 0 \mod p$:

$z_3 z_5^3+z_2 z_5^3+z_1 z_5^3+z_5^3+2 z_3 z_4 z_5^2+2 z_2 z_4 z_5^2+2 z_1 z_4 z_5^2+2 z_4 z_5^2+2 z_3^2 z_5^2+3 z_2 z_3 z_5^2+2 z_1 z_3 z_5^2+10 z_3 z_5^2+z_2^2 z_5^2+z_1 z_2 z_5^2+8 z_2 z_5^2+6 z_1 z_5^2+8 z_5^2+z_3 z_4^2 z_5+z_2 z_4^2 z_5+z_1 z_4^2 z_5+z_4^2 z_5+2 z_3^2 z_4 z_5+3 z_2 z_3 z_4 z_5+2 z_1 z_3 z_4 z_5+11 z_3 z_4 z_5+z_2^2 z_4 z_5+z_1 z_2 z_4 z_5+9 z_2 z_4 z_5+7 z_1 z_4 z_5+9 z_4 z_5+z_3^3 z_5+2 z_2 z_3^2 z_5+z_1 z_3^2 z_5+10 z_3^2 z_5+z_2^2 z_3 z_5+z_1 z_2 z_3 z_5+13 z_2 z_3 z_5+7 z_1 z_3 z_5+28 z_3 z_5+3 z_2^2 z_5+3 z_1 z_2 z_5+18 z_2 z_5+11 z_1 z_5+19 z_5+z_3 z_4^2+z_2 z_4^2+z_1 z_4^2+z_4^2+2 z_3^2 z_4+3 z_2 z_3 z_4+2 z_1 z_3 z_4+9 z_3 z_4+z_2^2 z_4+z_1 z_2 z_4+7 z_2 z_4+5 z_1 z_4+7 z_4+z_3^3+2 z_2 z_3^2+z_1 z_3^2+8 z_3^2+z_2^2 z_3+z_1 z_2 z_3+10 z_2 z_3+5 z_1 z_3+19 z_3+2 z_2^2+2 z_1 z_2+11 z_2+6 z_1+12 = 0$

\smallskip

\item $\ext(\phi\rtimes(01,01)\bigr) \not \equiv 0 \mod p$:

$z_3 z_5^3+z_2 z_5^3+z_1 z_5^3+z_5^3+2 z_3 z_4 z_5^2+2 z_2 z_4 z_5^2+2 z_1 z_4 z_5^2+2 z_4 z_5^2+2 z_3^2 z_5^2+3 z_2 z_3 z_5^2+2 z_1 z_3 z_5^2+8 z_3 z_5^2+z_2^2 z_5^2+z_1 z_2 z_5^2+6 z_2 z_5^2+4 z_1 z_5^2+6 z_5^2+z_3 z_4^2 z_5+z_2 z_4^2 z_5+z_1 z_4^2 z_5+z_4^2 z_5+2 z_3^2 z_4 z_5+3 z_2 z_3 z_4 z_5+2 z_1 z_3 z_4 z_5+8 z_3 z_4 z_5+z_2^2 z_4 z_5+z_1 z_2 z_4 z_5+6 z_2 z_4 z_5+4 z_1 z_4 z_5+6 z_4 z_5+z_3^3 z_5+2 z_2 z_3^2 z_5+z_1 z_3^2 z_5+7 z_3^2 z_5+z_2^2 z_3 z_5+z_1 z_2 z_3 z_5+8 z_2 z_3 z_5+4 z_1 z_3 z_5+14 z_3 z_5+z_2^2 z_5+z_1 z_2 z_5+6 z_2 z_5+3 z_1 z_5+8 z_5 \neq 0$

\smallskip

\item $\ext(\phi\rtimes(00,00)\bigr) \not \equiv 0 \mod p$:

$2 z_5^3+4 z_4 z_5^2+4 z_3 z_5^2+2 z_2 z_5^2+12 z_5^2+2 z_4^2 z_5+4 z_3 z_4 z_5+2 z_2 z_4 z_5+13 z_4 z_5+2 z_3^2 z_5+2 z_2 z_3 z_5+13 z_3 z_5+4 z_2 z_5+18 z_5+z_4^2+2 z_3 z_4+z_2 z_4+6 z_4+z_3^2+z_2 z_3+6 z_3+2 z_2+8 \neq 0$

\smallskip

\item $-2 \ext\bigl(M(\phi_\circ)\rtimes(10,11)\bigr) + \ext\bigl(L(\phi_\circ)\rtimes(10,11)\bigr) + \ext\bigl(R(\phi_\circ)\rtimes(10,11)\bigr) \equiv 0 \mod{p}$:

$z_5^4+3 z_4 z_5^3+3 z_3 z_5^3+2 z_2 z_5^3+z_1 z_5^3+6 z_5^3+3 z_4^2 z_5^2+6 z_3 z_4 z_5^2+4 z_2 z_4 z_5^2+2 z_1 z_4 z_5^2+12 z_4 z_5^2+3 z_3^2 z_5^2+4 z_2 z_3 z_5^2+2 z_1 z_3 z_5^2+12 z_3 z_5^2+z_2^2 z_5^2+z_1 z_2 z_5^2+7 z_2 z_5^2+3 z_1 z_5^2+11 z_5^2+z_4^3 z_5+3 z_3 z_4^2 z_5+2 z_2 z_4^2 z_5+z_1 z_4^2 z_5+6 z_4^2 z_5+3 z_3^2 z_4 z_5+4 z_2 z_3 z_4 z_5+2 z_1 z_3 z_4 z_5+12 z_3 z_4 z_5+z_2^2 z_4 z_5+z_1 z_2 z_4 z_5+7 z_2 z_4 z_5+3 z_1 z_4 z_5+11 z_4 z_5+z_3^3 z_5+2 z_2 z_3^2 z_5+z_1 z_3^2 z_5+6 z_3^2 z_5+z_2^2 z_3 z_5+z_1 z_2 z_3 z_5+7 z_2 z_3 z_5+3 z_1 z_3 z_5+11 z_3 z_5+z_2^2 z_5+z_1 z_2 z_5+5 z_2 z_5+2 z_1 z_5+6 z_5 = 0$

\smallskip

\item $-2 \ext\bigl(M(\phi_\circ)\rtimes(01,11)\bigr) + \ext\bigl(L(\phi_\circ)\rtimes(01,11)\bigr) + \ext\bigl(R(\phi_\circ)\rtimes(01,11)\bigr) \equiv 0 \mod{p}$:

$z_3 z_5^4+z_2 z_5^4+z_1 z_5^4+z_5^4+3 z_3 z_4 z_5^3+3 z_2 z_4 z_5^3+3 z_1 z_4 z_5^3+3 z_4 z_5^3+3 z_3^2 z_5^3+5 z_2 z_3 z_5^3+4 z_1 z_3 z_5^3+12 z_3 z_5^3+2 z_2^2 z_5^3+3 z_1 z_2 z_5^3+10 z_2 z_5^3+z_1^2 z_5^3+8 z_1 z_5^3+9 z_5^3+3 z_3 z_4^2 z_5^2+3 z_2 z_4^2 z_5^2+3 z_1 z_4^2 z_5^2+3 z_4^2 z_5^2+6 z_3^2 z_4 z_5^2+10 z_2 z_3 z_4 z_5^2+8 z_1 z_3 z_4 z_5^2+24 z_3 z_4 z_5^2+4 z_2^2 z_4 z_5^2+6 z_1 z_2 z_4 z_5^2+20 z_2 z_4 z_5^2+2 z_1^2 z_4 z_5^2+16 z_1 z_4 z_5^2+18 z_4 z_5^2+3 z_3^3 z_5^2+7 z_2 z_3^2 z_5^2+5 z_1 z_3^2 z_5^2+21 z_3^2 z_5^2+5 z_2^2 z_3 z_5^2+7 z_1 z_2 z_3 z_5^2+31 z_2 z_3 z_5^2+2 z_1^2 z_3 z_5^2+21 z_1 z_3 z_5^2+44 z_3 z_5^2+z_2^3 z_5^2+2 z_1 z_2^2 z_5^2+10 z_2^2 z_5^2+z_1^2 z_2 z_5^2+13 z_1 z_2 z_5^2+30 z_2 z_5^2+3 z_1^2 z_5^2+19 z_1 z_5^2+26 z_5^2+z_3 z_4^3 z_5+z_2 z_4^3 z_5+z_1 z_4^3 z_5+z_4^3 z_5+3 z_3^2 z_4^2 z_5+5 z_2 z_3 z_4^2 z_5+4 z_1 z_3 z_4^2 z_5+12 z_3 z_4^2 z_5+2 z_2^2 z_4^2 z_5+3 z_1 z_2 z_4^2 z_5+10 z_2 z_4^2 z_5+z_1^2 z_4^2 z_5+8 z_1 z_4^2 z_5+9 z_4^2 z_5+3 z_3^3 z_4 z_5+7 z_2 z_3^2 z_4 z_5+5 z_1 z_3^2 z_4 z_5+21 z_3^2 z_4 z_5+5 z_2^2 z_3 z_4 z_5+7 z_1 z_2 z_3 z_4 z_5+31 z_2 z_3 z_4 z_5+2 z_1^2 z_3 z_4 z_5+21 z_1 z_3 z_4 z_5+44 z_3 z_4 z_5+z_2^3 z_4 z_5+2 z_1 z_2^2 z_4 z_5+10 z_2^2 z_4 z_5+z_1^2 z_2 z_4 z_5+13 z_1 z_2 z_4 z_5+30 z_2 z_4 z_5+3 z_1^2 z_4 z_5+19 z_1 z_4 z_5+26 z_4 z_5+z_3^4 z_5+3 z_2 z_3^3 z_5+2 z_1 z_3^3 z_5+10 z_3^3 z_5+3 z_2^2 z_3^2 z_5+4 z_1 z_2 z_3^2 z_5+21 z_2 z_3^2 z_5+z_1^2 z_3^2 z_5+13 z_1 z_3^2 z_5+35 z_3^2 z_5+z_2^3 z_3 z_5+2 z_1 z_2^2 z_3 z_5+12 z_2^2 z_3 z_5+z_1^2 z_2 z_3 z_5+15 z_1 z_2 z_3 z_5+44 z_2 z_3 z_5+3 z_1^2 z_3 z_5+25 z_1 z_3 z_5+50 z_3 z_5+z_2^3 z_5+2 z_1 z_2^2 z_5+9 z_2^2 z_5+z_1^2 z_2 z_5+11 z_1 z_2 z_5+26 z_2 z_5+2 z_1^2 z_5+14 z_1 z_5+24 z_5 = 0$

\smallskip

\item$-2 \ext\bigl(M(\phi_\circ)\rtimes(00,01)\bigr) + \ext\bigl(L(\phi_\circ)\rtimes(00,01)\bigr) + \ext\bigl(R(\phi_\circ)\rtimes(00,01)\bigr) \equiv 0 \mod{p}$:

$z_5^4+3 z_4 z_5^3+3 z_3 z_5^3+2 z_2 z_5^3+z_1 z_5^3+8 z_5^3+3 z_4^2 z_5^2+6 z_3 z_4 z_5^2+4 z_2 z_4 z_5^2+2 z_1 z_4 z_5^2+16 z_4 z_5^2+3 z_3^2 z_5^2+4 z_2 z_3 z_5^2+2 z_1 z_3 z_5^2+16 z_3 z_5^2+z_2^2 z_5^2+z_1 z_2 z_5^2+9 z_2 z_5^2+4 z_1 z_5^2+19 z_5^2+z_4^3 z_5+3 z_3 z_4^2 z_5+2 z_2 z_4^2 z_5+z_1 z_4^2 z_5+8 z_4^2 z_5+3 z_3^2 z_4 z_5+4 z_2 z_3 z_4 z_5+2 z_1 z_3 z_4 z_5+16 z_3 z_4 z_5+z_2^2 z_4 z_5+z_1 z_2 z_4 z_5+9 z_2 z_4 z_5+4 z_1 z_4 z_5+19 z_4 z_5+z_3^3 z_5+2 z_2 z_3^2 z_5+z_1 z_3^2 z_5+8 z_3^2 z_5+z_2^2 z_3 z_5+z_1 z_2 z_3 z_5+9 z_2 z_3 z_5+4 z_1 z_3 z_5+19 z_3 z_5+z_2^2 z_5+z_1 z_2 z_5+7 z_2 z_5+3 z_1 z_5+12 z_5 = 0$

\smallskip

\item $-2 \ext\bigl(M(\phi_\circ)\rtimes(00,10)\bigr) + \ext\bigl(L(\phi_\circ)\rtimes(00,10)\bigr) + \ext\bigl(R(\phi_\circ)\rtimes(00,10)\bigr) \equiv 0 \mod{p}$:

$z_5^4+3 z_4 z_5^3+3 z_3 z_5^3+2 z_2 z_5^3+z_1 z_5^3+10 z_5^3+3 z_4^2 z_5^2+6 z_3 z_4 z_5^2+4 z_2 z_4 z_5^2+2 z_1 z_4 z_5^2+21 z_4 z_5^2+3 z_3^2 z_5^2+4 z_2 z_3 z_5^2+2 z_1 z_3 z_5^2+21 z_3 z_5^2+z_2^2 z_5^2+z_1 z_2 z_5^2+13 z_2 z_5^2+6 z_1 z_5^2+35 z_5^2+z_4^3 z_5+3 z_3 z_4^2 z_5+2 z_2 z_4^2 z_5+z_1 z_4^2 z_5+12 z_4^2 z_5+3 z_3^2 z_4 z_5+4 z_2 z_3 z_4 z_5+2 z_1 z_3 z_4 z_5+24 z_3 z_4 z_5+z_2^2 z_4 z_5+z_1 z_2 z_4 z_5+15 z_2 z_4 z_5+7 z_1 z_4 z_5+44 z_4 z_5+z_3^3 z_5+2 z_2 z_3^2 z_5+z_1 z_3^2 z_5+12 z_3^2 z_5+z_2^2 z_3 z_5+z_1 z_2 z_3 z_5+15 z_2 z_3 z_5+7 z_1 z_3 z_5+44 z_3 z_5+3 z_2^2 z_5+3 z_1 z_2 z_5+25 z_2 z_5+11 z_1 z_5+50 z_5+z_4^3+3 z_3 z_4^2+2 z_2 z_4^2+z_1 z_4^2+9 z_4^2+3 z_3^2 z_4+4 z_2 z_3 z_4+2 z_1 z_3 z_4+18 z_3 z_4+z_2^2 z_4+z_1 z_2 z_4+11 z_2 z_4+5 z_1 z_4+26 z_4+z_3^3+2 z_2 z_3^2+z_1 z_3^2+9 z_3^2+z_2^2 z_3+z_1 z_2 z_3+11 z_2 z_3+5 z_1 z_3+26 z_3+2 z_2^2+2 z_1 z_2+14 z_2+6 z_1+24 = 0$

\smallskip

\item $2\ext\bigl(M^2(\phi_\circ)\rtimes(00,11)\bigr) - 4 \ext\bigl(LM(\phi_\circ)\rtimes(00,11)\bigr) - 4\ext\bigl(RM(\phi_\circ)\rtimes(00,11)\bigr) +$

$\ext\bigl(L^2(\phi_\circ)\rtimes(00,11)\bigr) + 2\ext\bigl(LR(\phi_\circ)\rtimes(00,11)\bigr) + \ext\bigl(R^2(\phi_\circ)\rtimes(00,11)\bigr) \equiv 0 \mod{p}$:

$z_5^5+4 z_4 z_5^4+4 z_3 z_5^4+3 z_2 z_5^4+2 z_1 z_5^4+10 z_5^4+6 z_4^2 z_5^3+12 z_3 z_4 z_5^3+9 z_2 z_4 z_5^3+6 z_1 z_4 z_5^3+30 z_4 z_5^3+6 z_3^2 z_5^3+9 z_2 z_3 z_5^3+6 z_1 z_3 z_5^3+30 z_3 z_5^3+3 z_2^2 z_5^3+4 z_1 z_2 z_5^3+21 z_2 z_5^3+z_1^2 z_5^3+13 z_1 z_5^3+35 z_5^3+4 z_4^3 z_5^2+12 z_3 z_4^2 z_5^2+9 z_2 z_4^2 z_5^2+6 z_1 z_4^2 z_5^2+30 z_4^2 z_5^2+12 z_3^2 z_4 z_5^2+18 z_2 z_3 z_4 z_5^2+12 z_1 z_3 z_4 z_5^2+60 z_3 z_4 z_5^2+6 z_2^2 z_4 z_5^2+8 z_1 z_2 z_4 z_5^2+42 z_2 z_4 z_5^2+2 z_1^2 z_4 z_5^2+26 z_1 z_4 z_5^2+70 z_4 z_5^2+4 z_3^3 z_5^2+9 z_2 z_3^2 z_5^2+6 z_1 z_3^2 z_5^2+30 z_3^2 z_5^2+6 z_2^2 z_3 z_5^2+8 z_1 z_2 z_3 z_5^2+42 z_2 z_3 z_5^2+2 z_1^2 z_3 z_5^2+26 z_1 z_3 z_5^2+70 z_3 z_5^2+z_2^3 z_5^2+2 z_1 z_2^2 z_5^2+12 z_2^2 z_5^2+z_1^2 z_2 z_5^2+15 z_1 z_2 z_5^2+44 z_2 z_5^2+3 z_1^2 z_5^2+25 z_1 z_5^2+50 z_5^2+1 (z_4^4) z_5+4 z_3 z_4^3 z_5+3 z_2 z_4^3 z_5+2 z_1 z_4^3 z_5+10 z_4^3 z_5+6 z_3^2 z_4^2 z_5+9 z_2 z_3 z_4^2 z_5+6 z_1 z_3 z_4^2 z_5+30 z_3 z_4^2 z_5+3 z_2^2 z_4^2 z_5+4 z_1 z_2 z_4^2 z_5+21 z_2 z_4^2 z_5+z_1^2 z_4^2 z_5+13 z_1 z_4^2 z_5+35 z_4^2 z_5+4 z_3^3 z_4 z_5+9 z_2 z_3^2 z_4 z_5+6 z_1 z_3^2 z_4 z_5+30 z_3^2 z_4 z_5+6 z_2^2 z_3 z_4 z_5+8 z_1 z_2 z_3 z_4 z_5+42 z_2 z_3 z_4 z_5+2 z_1^2 z_3 z_4 z_5+26 z_1 z_3 z_4 z_5+70 z_3 z_4 z_5+z_2^3 z_4 z_5+2 z_1 z_2^2 z_4 z_5+12 z_2^2 z_4 z_5+z_1^2 z_2 z_4 z_5+15 z_1 z_2 z_4 z_5+44 z_2 z_4 z_5+3 z_1^2 z_4 z_5+25 z_1 z_4 z_5+50 z_4 z_5+z_3^4 z_5+3 z_2 z_3^3 z_5+2 z_1 z_3^3 z_5+10 z_3^3 z_5+3 z_2^2 z_3^2 z_5+4 z_1 z_2 z_3^2 z_5+21 z_2 z_3^2 z_5+z_1^2 z_3^2 z_5+13 z_1 z_3^2 z_5+35 z_3^2 z_5+z_2^3 z_3 z_5+2 z_1 z_2^2 z_3 z_5+12 z_2^2 z_3 z_5+z_1^2 z_2 z_3 z_5+15 z_1 z_2 z_3 z_5+44 z_2 z_3 z_5+3 z_1^2 z_3 z_5+25 z_1 z_3 z_5+50 z_3 z_5+z_2^3 z_5+2 z_1 z_2^2 z_5+9 z_2^2 z_5+z_1^2 z_2 z_5+11 z_1 z_2 z_5+26 z_2 z_5+2 z_1^2 z_5+14 z_1 z_5+24 z_5 = 0$

\end{enumerate}
\medskip

\textsc{SOLUTION}:$\.(z_1,z_2,z_3,z_4,z_5) = (-2,1,-3,1,-1).$

The nonzero values~$\ext(\sigma\rtimes(v,v'))$ takes are~$2$ and~$4$.

\end{footnotesize}

\bigskip

\subsection{\normalsize \textsc{TestEq} gate.}
\label{TestEqeqn}

\begin{footnotesize}

\begin{enumerate}

\item $|\phi|-2 \equiv 0 \mod{p^3}$:

$z_1+z_2+z_3+z_4+ z_5+4 = 0$

\smallskip

\item $\ext\bigl(\phi\rtimes(11,11)\bigr) \not \equiv 0 \mod p$:

$2 z_3 z_5^3+2 z_2 z_5^3+2 z_1 z_5^3+4 z_5^3+6 z_3 z_4 z_5^2+6 z_2 z_4 z_5^2+6 z_1 z_4 z_5^2+12 z_4 z_5^2+6 z_3^2 z_5^2+9 z_2 z_3 z_5^2+6 z_1 z_3 z_5^2+24 z_3 z_5^2+3 z_2^2 z_5^2+3 z_1 z_2 z_5^2+15 z_2 z_5^2+6 z_1 z_5^2+18 z_5^2+6 z_3 z_4^2 z_5+6 z_2 z_4^2 z_5+6 z_1 z_4^2 z_5+12 z_4^2 z_5+12 z_3^2 z_4 z_5+18 z_2 z_3 z_4 z_5+12 z_1 z_3 z_4 z_5+48 z_3 z_4 z_5+6 z_2^2 z_4 z_5+6 z_1 z_2 z_4 z_5+30 z_2 z_4 z_5+12 z_1 z_4 z_5+36 z_4 z_5+6 z_3^3 z_5+12 z_2 z_3^2 z_5+6 z_1 z_3^2 z_5+36 z_3^2 z_5+6 z_2^2 z_3 z_5+6 z_1 z_2 z_3 z_5+39 z_2 z_3 z_5+12 z_1 z_3 z_5+58 z_3 z_5+3 z_2^2 z_5+3 z_1 z_2 z_5+19 z_2 z_5+4 z_1 z_5+26 z_5 \neq 0$

\smallskip

\item $\ext\bigl(\phi\rtimes(10,01)\bigr) \equiv 0 \mod p$:

$2 z_5^3+6 z_4 z_5^2+6 z_3 z_5^2+3 z_2 z_5^2+12 z_5^2+6 z_4^2 z_5+12 z_3 z_4 z_5+6 z_2 z_4 z_5+24 z_4 z_5+6 z_3^2 z_5+6 z_2 z_3 z_5+24 z_3 z_5+9 z_2 z_5+16 z_5+6 z_4^2+12 z_3 z_4+6 z_2 z_4+12 z_4+6 z_3^2+6 z_2 z_3+12 z_3 = 0$

\smallskip

\item $\ext\bigl(\phi\rtimes(10,10)\bigr)  \equiv 0 \mod p$:

$2 z_5^3+6 z_4 z_5^2+6 z_3 z_5^2+3 z_2 z_5^2+12 z_5^2+6 z_4^2 z_5+12 z_3 z_4 z_5+6 z_2 z_4 z_5+24 z_4 z_5+6 z_3^2 z_5+6 z_2 z_3 z_5+24 z_3 z_5+9 z_2 z_5+22 z_5+6 z_4^2+12 z_3 z_4+6 z_2 z_4+18 z_4+6 z_3^2+6 z_2 z_3+18 z_3+6 z_2+12 = 0$

\smallskip

\item $\ext\bigl(\phi\rtimes(01,10)\bigr)  \equiv 0 \mod p$:

$2 z_3 z_5^3+2 z_2 z_5^3+2 z_1 z_5^3+2 z_5^3+6 z_3 z_4 z_5^2+6 z_2 z_4 z_5^2+6 z_1 z_4 z_5^2+6 z_4 z_5^2+6 z_3^2 z_5^2+9 z_2 z_3 z_5^2+6 z_1 z_3 z_5^2+24 z_3 z_5^2+3 z_2^2 z_5^2+3 z_1 z_2 z_5^2+18 z_2 z_5^2+12 z_1 z_5^2+18 z_5^2+6 z_3 z_4^2 z_5+6 z_2 z_4^2 z_5+6 z_1 z_4^2 z_5+6 z_4^2 z_5+12 z_3^2 z_4 z_5+18 z_2 z_3 z_4 z_5+12 z_1 z_3 z_4 z_5+48 z_3 z_4 z_5+6 z_2^2 z_4 z_5+6 z_1 z_2 z_4 z_5+36 z_2 z_4 z_5+24 z_1 z_4 z_5+36 z_4 z_5+6 z_3^3 z_5+12 z_2 z_3^2 z_5+6 z_1 z_3^2 z_5+42 z_3^2 z_5+6 z_2^2 z_3 z_5+6 z_1 z_2 z_3 z_5+51 z_2 z_3 z_5+24 z_1 z_3 z_5+88 z_3 z_5+9 z_2^2 z_5+9 z_1 z_2 z_5+46 z_2 z_5+22 z_1 z_5+52 z_5+6 z_3 z_4^2+6 z_2 z_4^2+6 z_1 z_4^2+6 z_4^2+12 z_3^2 z_4+18 z_2 z_3 z_4+12 z_1 z_3 z_4+42 z_3 z_4+6 z_2^2 z_4+6 z_1 z_2 z_4+30 z_2 z_4+18 z_1 z_4+30 z_4+6 z_3^3+12 z_2 z_3^2+6 z_1 z_3^2+36 z_3^2+6 z_2^2 z_3+6 z_1 z_2 z_3+42 z_2 z_3+18 z_1 z_3+66 z_3+6 z_2^2+6 z_1 z_2+30 z_2+12 z_1+36 = 0$

\smallskip

\item $\ext(\phi\rtimes(01,01)\bigr) \equiv 0 \mod p$:

$2 z_3 z_5^3+2 z_2 z_5^3+2 z_1 z_5^3+2 z_5^3+6 z_3 z_4 z_5^2+6 z_2 z_4 z_5^2+6 z_1 z_4 z_5^2+6 z_4 z_5^2+6 z_3^2 z_5^2+9 z_2 z_3 z_5^2+6 z_1 z_3 z_5^2+24 z_3 z_5^2+3 z_2^2 z_5^2+3 z_1 z_2 z_5^2+18 z_2 z_5^2+12 z_1 z_5^2+18 z_5^2+6 z_3 z_4^2 z_5+6 z_2 z_4^2 z_5+6 z_1 z_4^2 z_5+6 z_4^2 z_5+12 z_3^2 z_4 z_5+18 z_2 z_3 z_4 z_5+12 z_1 z_3 z_4 z_5+48 z_3 z_4 z_5+6 z_2^2 z_4 z_5+6 z_1 z_2 z_4 z_5+36 z_2 z_4 z_5+24 z_1 z_4 z_5+36 z_4 z_5+6 z_3^3 z_5+12 z_2 z_3^2 z_5+6 z_1 z_3^2 z_5+42 z_3^2 z_5+6 z_2^2 z_3 z_5+6 z_1 z_2 z_3 z_5+51 z_2 z_3 z_5+24 z_1 z_3 z_5+82 z_3 z_5+9 z_2^2 z_5+9 z_1 z_2 z_5+40 z_2 z_5+16 z_1 z_5+46 z_5+6 z_3 z_4^2+6 z_2 z_4^2+6 z_1 z_4^2+6 z_4^2+12 z_3^2 z_4+18 z_2 z_3 z_4+12 z_1 z_3 z_4+36 z_3 z_4+6 z_2^2 z_4+6 z_1 z_2 z_4+24 z_2 z_4+12 z_1 z_4+24 z_4+6 z_3^3+12 z_2 z_3^2+6 z_1 z_3^2+30 z_3^2+6 z_2^2 z_3+6 z_1 z_2 z_3+30 z_2 z_3+12 z_1 z_3+42 z_3+6 z_2+18 = 0$

\smallskip

\item $\ext(\phi\rtimes(00,00)\bigr) \not \equiv 0 \mod p$:

$2 z_5^3+6 z_4 z_5^2+6 z_3 z_5^2+3 z_2 z_5^2+18 z_5^2+6 z_4^2 z_5+12 z_3 z_4 z_5+6 z_2 z_4 z_5+36 z_4 z_5+6 z_3^2 z_5+6 z_2 z_3 z_5+36 z_3 z_5+15 z_2 z_5+40 z_5+12 z_4^2+24 z_3 z_4+12 z_2 z_4+36 z_4+12 z_3^2+12 z_2 z_3+36 z_3+6 z_2+15 \neq 0$

\smallskip

\item $-2 \ext\bigl(M(\phi_\circ)\rtimes(10,11)\bigr) + \ext\bigl(L(\phi_\circ)\rtimes(10,11)\bigr) + \ext\bigl(R(\phi_\circ)\rtimes(10,11)\bigr) \equiv 0 \mod{p}$:

$2 z_5^4+8 z_4 z_5^3+8 z_3 z_5^3+5 z_2 z_5^3+2 z_1 z_5^3+12 z_5^3+12 z_4^2 z_5^2+24 z_3 z_4 z_5^2+15 z_2 z_4 z_5^2+6 z_1 z_4 z_5^2+36 z_4 z_5^2+12 z_3^2 z_5^2+15 z_2 z_3 z_5^2+6 z_1 z_3 z_5^2+36 z_3 z_5^2+3 z_2^2 z_5^2+3 z_1 z_2 z_5^2+18 z_2 z_5^2+6 z_1 z_5^2+22 z_5^2+6 z_4^3 z_5+18 z_3 z_4^2 z_5+12 z_2 z_4^2 z_5+6 z_1 z_4^2 z_5+30 z_4^2 z_5+18 z_3^2 z_4 z_5+24 z_2 z_3 z_4 z_5+12 z_1 z_3 z_4 z_5+60 z_3 z_4 z_5+6 z_2^2 z_4 z_5+6 z_1 z_2 z_4 z_5+33 z_2 z_4 z_5+12 z_1 z_4 z_5+40 z_4 z_5+6 z_3^3 z_5+12 z_2 z_3^2 z_5+6 z_1 z_3^2 z_5+30 z_3^2 z_5+6 z_2^2 z_3 z_5+6 z_1 z_2 z_3 z_5+33 z_2 z_3 z_5+12 z_1 z_3 z_5+40 z_3 z_5+3 z_2^2 z_5+3 z_1 z_2 z_5+13 z_2 z_5+4 z_1 z_5+12 z_5 = 0$

\smallskip

\item $-2 \ext\bigl(M(\phi_\circ)\rtimes(01,11)\bigr) + \ext\bigl(L(\phi_\circ)\rtimes(01,11)\bigr) + \ext\bigl(R(\phi_\circ)\rtimes(01,11)\bigr) \equiv 0 \mod{p}$:

$2 z_3 z_5^4+2 z_2 z_5^4+2 z_1 z_5^4+2 z_5^4+8 z_3 z_4 z_5^3+8 z_2 z_4 z_5^3+8 z_1 z_4 z_5^3+8 z_4 z_5^3+8 z_3^2 z_5^3+13 z_2 z_3 z_5^3+10 z_1 z_3 z_5^3+28 z_3 z_5^3+5 z_2^2 z_5^3+7 z_1 z_2 z_5^3+22 z_2 z_5^3+2 z_1^2 z_5^3+16 z_1 z_5^3+20 z_5^3+12 z_3 z_4^2 z_5^2+12 z_2 z_4^2 z_5^2+12 z_1 z_4^2 z_5^2+12 z_4^2 z_5^2+24 z_3^2 z_4 z_5^2+39 z_2 z_3 z_4 z_5^2+30 z_1 z_3 z_4 z_5^2+84 z_3 z_4 z_5^2+15 z_2^2 z_4 z_5^2+21 z_1 z_2 z_4 z_5^2+66 z_2 z_4 z_5^2+6 z_1^2 z_4 z_5^2+48 z_1 z_4 z_5^2+60 z_4 z_5^2+12 z_3^3 z_5^2+27 z_2 z_3^2 z_5^2+18 z_1 z_3^2 z_5^2+72 z_3^2 z_5^2+18 z_2^2 z_3 z_5^2+24 z_1 z_2 z_3 z_5^2+99 z_2 z_3 z_5^2+6 z_1^2 z_3 z_5^2+60 z_1 z_3 z_5^2+130 z_3 z_5^2+3 z_2^3 z_5^2+6 z_1 z_2^2 z_5^2+27 z_2^2 z_5^2+3 z_1^2 z_2 z_5^2+33 z_1 z_2 z_5^2+76 z_2 z_5^2+6 z_1^2 z_5^2+40 z_1 z_5^2+70 z_5^2+6 z_3 z_4^3 z_5+6 z_2 z_4^3 z_5+6 z_1 z_4^3 z_5+6 z_4^3 z_5+18 z_3^2 z_4^2 z_5+30 z_2 z_3 z_4^2 z_5+24 z_1 z_3 z_4^2 z_5+66 z_3 z_4^2 z_5+12 z_2^2 z_4^2 z_5+18 z_1 z_2 z_4^2 z_5+54 z_2 z_4^2 z_5+6 z_1^2 z_4^2 z_5+42 z_1 z_4^2 z_5+48 z_4^2 z_5+18 z_3^3 z_4 z_5+42 z_2 z_3^2 z_4 z_5+30 z_1 z_3^2 z_4 z_5+114 z_3^2 z_4 z_5+30 z_2^2 z_3 z_4 z_5+42 z_1 z_2 z_3 z_4 z_5+165 z_2 z_3 z_4 z_5+12 z_1^2 z_3 z_4 z_5+108 z_1 z_3 z_4 z_5+214 z_3 z_4 z_5+6 z_2^3 z_4 z_5+12 z_1 z_2^2 z_4 z_5+51 z_2^2 z_4 z_5+6 z_1^2 z_2 z_4 z_5+63 z_1 z_2 z_4 z_5+136 z_2 z_4 z_5+12 z_1^2 z_4 z_5+76 z_1 z_4 z_5+118 z_4 z_5+6 z_3^4 z_5+18 z_2 z_3^3 z_5+12 z_1 z_3^3 z_5+54 z_3^3 z_5+18 z_2^2 z_3^2 z_5+24 z_1 z_2 z_3^2 z_5+111 z_2 z_3^2 z_5+6 z_1^2 z_3^2 z_5+66 z_1 z_3^2 z_5+166 z_3^2 z_5+6 z_2^3 z_3 z_5+12 z_1 z_2^2 z_3 z_5+60 z_2^2 z_3 z_5+6 z_1^2 z_2 z_3 z_5+72 z_1 z_2 z_3 z_5+194 z_2 z_3 z_5+12 z_1^2 z_3 z_5+98 z_1 z_3 z_5+206 z_3 z_5+3 z_2^3 z_5+6 z_1 z_2^2 z_5+28 z_2^2 z_5+3 z_1^2 z_2 z_5+32 z_1 z_2 z_5+86 z_2 z_5+4 z_1^2 z_5+38 z_1 z_5+88 z_5 = 0$

\smallskip

\item$-2 \ext\bigl(M(\phi_\circ)\rtimes(00,01)\bigr) + \ext\bigl(L(\phi_\circ)\rtimes(00,01)\bigr) + \ext\bigl(R(\phi_\circ)\rtimes(00,01)\bigr) \equiv 0 \mod{p}$:

$2 z_5^4+8 z_4 z_5^3+8 z_3 z_5^3+5 z_2 z_5^3+2 z_1 z_5^3+20 z_5^3+12 z_4^2 z_5^2+24 z_3 z_4 z_5^2+15 z_2 z_4 z_5^2+6 z_1 z_4 z_5^2+60 z_4 z_5^2+12 z_3^2 z_5^2+15 z_2 z_3 z_5^2+6 z_1 z_3 z_5^2+60 z_3 z_5^2+3 z_2^2 z_5^2+3 z_1 z_2 z_5^2+33 z_2 z_5^2+12 z_1 z_5^2+64 z_5^2+6 z_4^3 z_5+18 z_3 z_4^2 z_5+12 z_2 z_4^2 z_5+6 z_1 z_4^2 z_5+54 z_4^2 z_5+18 z_3^2 z_4 z_5+24 z_2 z_3 z_4 z_5+12 z_1 z_3 z_4 z_5+108 z_3 z_4 z_5+6 z_2^2 z_4 z_5+6 z_1 z_2 z_4 z_5+63 z_2 z_4 z_5+24 z_1 z_4 z_5+124 z_4 z_5+6 z_3^3 z_5+12 z_2 z_3^2 z_5+6 z_1 z_3^2 z_5+54 z_3^2 z_5+6 z_2^2 z_3 z_5+6 z_1 z_2 z_3 z_5+63 z_2 z_3 z_5+24 z_1 z_3 z_5+124 z_3 z_5+9 z_2^2 z_5+9 z_1 z_2 z_5+52 z_2 z_5+16 z_1 z_5+64 z_5+6 z_4^3+18 z_3 z_4^2+12 z_2 z_4^2+6 z_1 z_4^2+36 z_4^2+18 z_3^2 z_4+24 z_2 z_3 z_4+12 z_1 z_3 z_4+72 z_3 z_4+6 z_2^2 z_4+6 z_1 z_2 z_4+36 z_2 z_4+12 z_1 z_4+48 z_4+6 z_3^3+12 z_2 z_3^2+6 z_1 z_3^2+36 z_3^2+6 z_2^2 z_3+6 z_1 z_2 z_3+36 z_2 z_3+12 z_1 z_3+48 z_3 = 0$

\smallskip

\item $-2 \ext\bigl(M(\phi_\circ)\rtimes(00,10)\bigr) + \ext\bigl(L(\phi_\circ)\rtimes(00,10)\bigr) + \ext\bigl(R(\phi_\circ)\rtimes(00,10)\bigr) \equiv 0 \mod{p}$:

$2 z_5^4+8 z_4 z_5^3+8 z_3 z_5^3+5 z_2 z_5^3+2 z_1 z_5^3+20 z_5^3+12 z_4^2 z_5^2+24 z_3 z_4 z_5^2+15 z_2 z_4 z_5^2+6 z_1 z_4 z_5^2+60 z_4 z_5^2+12 z_3^2 z_5^2+15 z_2 z_3 z_5^2+6 z_1 z_3 z_5^2+60 z_3 z_5^2+3 z_2^2 z_5^2+3 z_1 z_2 z_5^2+33 z_2 z_5^2+12 z_1 z_5^2+70 z_5^2+6 z_4^3 z_5+18 z_3 z_4^2 z_5+12 z_2 z_4^2 z_5+6 z_1 z_4^2 z_5+54 z_4^2 z_5+18 z_3^2 z_4 z_5+24 z_2 z_3 z_4 z_5+12 z_1 z_3 z_4 z_5+108 z_3 z_4 z_5+6 z_2^2 z_4 z_5+6 z_1 z_2 z_4 z_5+63 z_2 z_4 z_5+24 z_1 z_4 z_5+136 z_4 z_5+6 z_3^3 z_5+12 z_2 z_3^2 z_5+6 z_1 z_3^2 z_5+54 z_3^2 z_5+6 z_2^2 z_3 z_5+6 z_1 z_2 z_3 z_5+63 z_2 z_3 z_5+24 z_1 z_3 z_5+136 z_3 z_5+9 z_2^2 z_5+9 z_1 z_2 z_5+64 z_2 z_5+22 z_1 z_5+100 z_5+6 z_4^3+18 z_3 z_4^2+12 z_2 z_4^2+6 z_1 z_4^2+42 z_4^2+18 z_3^2 z_4+24 z_2 z_3 z_4+12 z_1 z_3 z_4+84 z_3 z_4+6 z_2^2 z_4+6 z_1 z_2 z_4+48 z_2 z_4+18 z_1 z_4+84 z_4+6 z_3^3+12 z_2 z_3^2+6 z_1 z_3^2+42 z_3^2+6 z_2^2 z_3+6 z_1 z_2 z_3+48 z_2 z_3+18 z_1 z_3+84 z_3+6 z_2^2+6 z_1 z_2+36 z_2+12 z_1+48 = 0$

\smallskip

\item $2\ext\bigl(M^2(\phi_\circ)\rtimes(00,11)\bigr) - 4 \ext\bigl(LM(\phi_\circ)\rtimes(00,11)\bigr) - 4\ext\bigl(RM(\phi_\circ)\rtimes(00,11)\bigr) +$

$\ext\bigl(L^2(\phi_\circ)\rtimes(00,11)\bigr) + 2\ext\bigl(LR(\phi_\circ)\rtimes(00,11)\bigr) + \ext\bigl(R^2(\phi_\circ)\rtimes(00,11)\bigr) \equiv 0 \mod{p}$:

$2 z_5^5+10 z_4 z_5^4+10 z_3 z_5^4+7 z_2 z_5^4+4 z_1 z_5^4+20 z_5^4+20 z_4^2 z_5^3+40 z_3 z_4 z_5^3+28 z_2 z_4 z_5^3+16 z_1 z_4 z_5^3+80 z_4 z_5^3+20 z_3^2 z_5^3+28 z_2 z_3 z_5^3+16 z_1 z_3 z_5^3+80 z_3 z_5^3+8 z_2^2 z_5^3+10 z_1 z_2 z_5^3+50 z_2 z_5^3+2 z_1^2 z_5^3+26 z_1 z_5^3+70 z_5^3+18 z_4^3 z_5^2+54 z_3 z_4^2 z_5^2+39 z_2 z_4^2 z_5^2+24 z_1 z_4^2 z_5^2+114 z_4^2 z_5^2+54 z_3^2 z_4 z_5^2+78 z_2 z_3 z_4 z_5^2+48 z_1 z_3 z_4 z_5^2+228 z_3 z_4 z_5^2+24 z_2^2 z_4 z_5^2+30 z_1 z_2 z_4 z_5^2+147 z_2 z_4 z_5^2+6 z_1^2 z_4 z_5^2+78 z_1 z_4 z_5^2+206 z_4 z_5^2+18 z_3^3 z_5^2+39 z_2 z_3^2 z_5^2+24 z_1 z_3^2 z_5^2+114 z_3^2 z_5^2+24 z_2^2 z_3 z_5^2+30 z_1 z_2 z_3 z_5^2+147 z_2 z_3 z_5^2+6 z_1^2 z_3 z_5^2+78 z_1 z_3 z_5^2+206 z_3 z_5^2+3 z_2^3 z_5^2+6 z_1 z_2^2 z_5^2+33 z_2^2 z_5^2+3 z_1^2 z_2 z_5^2+39 z_1 z_2 z_5^2+107 z_2 z_5^2+6 z_1^2 z_5^2+50 z_1 z_5^2+100 z_5^2+6 (z_4^4) z_5+24 z_3 z_4^3 z_5+18 z_2 z_4^3 z_5+12 z_1 z_4^3 z_5+54 z_4^3 z_5+36 z_3^2 z_4^2 z_5+54 z_2 z_3 z_4^2 z_5+36 z_1 z_3 z_4^2 z_5+162 z_3 z_4^2 z_5+18 z_2^2 z_4^2 z_5+24 z_1 z_2 z_4^2 z_5+111 z_2 z_4^2 z_5+6 z_1^2 z_4^2 z_5+66 z_1 z_4^2 z_5+160 z_4^2 z_5+24 z_3^3 z_4 z_5+54 z_2 z_3^2 z_4 z_5+36 z_1 z_3^2 z_4 z_5+162 z_3^2 z_4 z_5+36 z_2^2 z_3 z_4 z_5+48 z_1 z_2 z_3 z_4 z_5+222 z_2 z_3 z_4 z_5+12 z_1^2 z_3 z_4 z_5+132 z_1 z_3 z_4 z_5+320 z_3 z_4 z_5+6 z_2^3 z_4 z_5+12 z_1 z_2^2 z_4 z_5+60 z_2^2 z_4 z_5+6 z_1^2 z_2 z_4 z_5+72 z_1 z_2 z_4 z_5+185 z_2 z_4 z_5+12 z_1^2 z_4 z_5+92 z_1 z_4 z_5+172 z_4 z_5+6 z_3^4 z_5+18 z_2 z_3^3 z_5+12 z_1 z_3^3 z_5+54 z_3^3 z_5+18 z_2^2 z_3^2 z_5+24 z_1 z_2 z_3^2 z_5+111 z_2 z_3^2 z_5+6 z_1^2 z_3^2 z_5+66 z_1 z_3^2 z_5+160 z_3^2 z_5+6 z_2^3 z_3 z_5+12 z_1 z_2^2 z_3 z_5+60 z_2^2 z_3 z_5+6 z_1^2 z_2 z_3 z_5+72 z_1 z_2 z_3 z_5+185 z_2 z_3 z_5+12 z_1^2 z_3 z_5+92 z_1 z_3 z_5+172 z_3 z_5+3 z_2^3 z_5+6 z_1 z_2^2 z_5+25 z_2^2 z_5+3 z_1^2 z_2 z_5+29 z_1 z_2 z_5+64 z_2 z_5+4 z_1^2 z_5+28 z_1 z_5+48 z_5 = 0$

\end{enumerate}
\medskip

\textsc{SOLUTION}:$\.(z_1,z_2,z_3,z_4,z_5) = (-2,-\tfrac{8}{3},\tfrac{5}{3},-3,2).$

The nonzero values~$\ext(\sigma\rtimes(v,v'))$ takes are~$\tfrac{7}{3}$ and~$\tfrac{-8}{3}$.

\end{footnotesize}

\end{document}